\documentclass[a4paper,reqno]{amsart}

\usepackage{amssymb}
\usepackage{amstext}
\usepackage{amsmath}
\usepackage{amscd}
\usepackage{amsthm}
\usepackage{amsfonts}
\usepackage{enumerate}
\usepackage{graphicx}
\usepackage{latexsym}
\usepackage[all]{xy}
\input xy
\xyoption{all}
\usepackage[usenames]{color}

\newtheorem{theorem}{Theorem}[section]

\newtheorem{corollary}[theorem]{Corollary}
\newtheorem{lemma}[theorem]{Lemma}
\newtheorem{proposition}[theorem]{Proposition}

\theoremstyle{definition}
\newtheorem{definition}[theorem]{Definition}
\newtheorem{remark}[theorem]{Remark}
\newtheorem{example}[theorem]{Example}
\newtheorem{question}[theorem]{Question}

\newcommand{\Mod}[1]{{\rm{Mod}}\text{-}#1} %Mod-
\newcommand{\PP}{{\mathcal P}}
\newcommand{\CC}{{\mathcal C}}
\newcommand{\DD}{{\mathcal D}}
\newcommand{\SSS}{{\mathcal S}}
\newcommand{\UU}{{\mathcal U}}
\newcommand{\VV}{{\mathcal V}}
\newcommand{\TT}{{\mathcal T}}
\newcommand{\MM}{{\mathcal M}}
\newcommand{\NN}{{\mathcal N}}
\newcommand{\KK}{{\mathcal K}}
\newcommand{\LL}{{\mathcal L}}
\newcommand{\XX}{{\mathcal X}}
\newcommand{\YY}{{\mathcal Y}}
\newcommand{\Z}{\mathbb{Z}}
\newcommand{\ind}{\mathsf{ind}}
\newcommand{\add}{\mathsf{add}}
\newcommand{\Add}{\mathsf{Add}}
\newcommand{\summand}{\mathsf{smd}}
\newcommand{\thick}{\mathsf{thick}}
\newcommand{\Hom}{\operatorname{Hom}\nolimits}
\newcommand{\silt}{\operatorname{silt}\nolimits}

\newcommand{\modu}[1]{{\rm{mod}}\text{-}#1} %mod-
\newcommand{\proj}[1]{{\rm{proj}}\text{-}#1} %proj-
\newcommand{\sta}[1]{{\underline{{\rm{mod}}}}\text{-}#1} %stable module cat
\newcommand{\homo}[3]{{\rm{Hom}}_{#1}(#2, #3)} %homo
\newcommand{\ehomo}[2]{{\rm{End}}_{#1}(#2)} %end
\newcommand{\fl}[1]{{\rm{f.l}}\text{-}#1} %f.l-
\newtheorem{assumption}[theorem]{Assumption}
\newcommand{\exc}{\operatorname{exp}\nolimits}
\newcommand{\K}{\mathsf{K}}
\newcommand{\D}{\mathsf{D}}
\newcommand{\End}{\operatorname{End}\nolimits}
\newcommand{\Ext}{\operatorname{Ext}\nolimits}
\renewcommand{\AA}{{\mathcal A}}

\newcommand{\RHom}{\mathbf{R}\operatorname{Hom}\nolimits}
\begin{document}

\title{Silting mutation in triangulated categories}
\author{Takuma Aihara and Osamu Iyama}
%\date{\today}
\address{T. Aihara: Division of Mathematical Science and Physics, 
Graduate School of Science and Technology, Chiba University, Yayoi-cho, Chiba 263-8522, Japan}
\email{taihara@math.s.chiba-u.ac.jp}
\address{O. Iyama: Graduate School of Mathematics, Nagoya University, Chikusa-ku, Nagoya,
464-8602 Japan}
\email{iyama@math.nagoya-u.ac.jp}
\thanks{2010 {\em Mathematics Subject Classification.} 16E30, 18E30}

\begin{abstract}
In representation theory of algebras the notion of `mutation' often plays important roles,
and two cases are well known, i.e. `cluster tilting mutation' and `exceptional mutation'.
In this paper we focus on `tilting mutation', which has a disadvantage that it is often 
impossible, i.e. some of summands of a tilting object can not be replaced to get a new tilting 
object. The aim of this paper is to take away this disadvantage by introducing `silting mutation'
for silting objects as a generalization of `tilting mutation'. We shall develop a basic
theory of silting mutation.
In particular, we introduce a partial order on the set of silting objects and establish the relationship with `silting mutation' by generalizing the theory of
Riedtmann-Schofield and Happel-Unger. We show that iterated silting mutation act transitively on the set of silting objects for local, hereditary or canonical algebras.
Finally we give a bijection between silting subcategories and certain t-structures.
\end{abstract}
\maketitle
\tableofcontents

%%%%%%%%%%%%%%%%%%%%%%%%%%%%%%%%%%%%%%%%%%%%%%%%%%%%%%%%%%%%%%%%%%%%%%%%%%%%%%%%%%%%%%%%%%%%%%%%%%%%%%%%%%%
%%%%%%%%%%%%%%%%%%%%%%%%%%%%%%%%%%%%%%%%%%%%%%%%%%%%%%%%%%%%%%%%%%%%%%%%%%%%%%%%%%%%%%%%%%%%%%%%%%%%%%%%%%%

\section{Introduction}

In representation theory of algebras the notion of `mutation' often plays important roles. 
Mutation is an operation for a certain class of objects (e.g. cluster tilting objects)
in a fixed category to construct a new object from a given one by replacing
a summand.
Two important cases are well-known: One is `cluster tilting mutation' \cite{BMRRT,IY,IW}
for cluster tilting objects which is applied for categorification of
Fomin-Zelevinsky cluster algebras,
and the other is `exceptional mutation' for exceptional sequences 
which is used to study the structure of derived categories of algebraic varieties \cite{GR,C}.

From Morita theoretic viewpoint \cite{Ric1}, tilting objects are the most important
class of objects, and `tilting mutation' for tilting objects has
been studied by several authors. `Tilting mutation' has the origin in
BGP (=Bernstein-Gelfand-Ponomarev) reflection functor \cite{BGP} in quiver
representation theory, and reflection functors are understood as a special class
of APR (=Auslander-Platzeck-Reiten) tilting modules \cite{APR},
which are `tilting mutations' obtained by replacing a simple summand of
the tilting $A$-modules $A$. In 1991 general `tilting mutation' for tilting modules was
introduced by Riedtmann-Schofield \cite{RS} in their study of combinatorial aspects
of tilting theory. It has been shown by Happel-Unger \cite{HU} that `tilting mutation'
is closely related to the partial order of tilting modules given by the inclusion
relation of the associated t-structures, and this is a big advantage of `tilting mutation'
which `cluster tilting mutation' and `exceptional mutation' do not have.
Their beautiful theory of `tilting mutation' has a lot of important applications, especially in the study of
cluster categories and preprojective algebras \cite{BMRRT,IR,BIRS,BIKR,SY}.
Another important source of `tilting mutation' is modular representation theory, 
where `tilting mutation' for tilting complexes was introduced by 
Okuyama and Rickard \cite{O,Ric2,HK}, and has played an important role 
in the study of Brou\'e's abelian defect group conjecture.

It is remarkable that `tilting mutation' has a big disadvantage
that it is often impossible, i.e. some of summands of a tilting
object can not be replaced to get a new tilting object. So it is usual that we can not
get sufficiently many tilting objects by iterated `tilting mutation'.
For example the tilting $A$-module $DA$ usually can not be obtained by
iterated `tilting mutation' of $A$ even for the case of hereditary algebras $A$.
The aim of this paper is to take away this disadvantage by introducing 
`\emph{silting mutation}' for \emph{silting objects}. In this context
`tilting mutation' should be understood as a special case of `silting mutation'.
Actually classical notion of APR and BB tilting modules and Okuyama-Rickard complexes
can be understood as special cases of `silting mutation'  (Theorems \ref{prop:APR is APR} and \ref{prop:1_right_mu}).
Silting objects are generalization of tilting objects,
and sometimes appeared in representation theory mainly in the study of
t-structures, e.g. Keller-Vossieck \cite{KV}, 
Hoshino-Kato-Miyachi \cite{HKM}, Assem-Salorio-Trepode \cite{AST} and Wei \cite{W}.
A point is that `silting mutation' is always possible in the sense that
any summand of a silting object always can be replaced to get a new
silting object (Theorem \ref{thm:mutation is silting}).  
Hence it is natural to hope that `silting mutation' gives us sufficiently many
silting objects in triangulated categories. We pose the following question
(Question \ref{transitive question2}),
where the corresponding property for tilting objects is usually not satisfied.

\begin{question}\label{transitive question}
Let $A$ be a finite dimensional algebra over a field.
When does $A$ satisfy the following property (T) (respectively, (T$'$))?

(T) (respectively, (T$'$)) The action of iterated irreducible `silting mutation'
(respectively, iterated `silting mutation') on the set of basic silting objects in $\K^{\rm b}(\proj{A})$ is transitive.
\end{question}

We shall show the following partial answers in Corollary \ref{corollary indecomposable case} and Theorem \ref{hereditary case}.

\begin{theorem}
If $A$ is either a local algebra, a hereditary algebra or a canonical algebra, then (T) is satisfied.
\end{theorem}

It will be shown in \cite{A2} that (T) is satisfied also for representation-finite symmetric algebras.
On the other hand, there is a symmetric algebra such that (T) is not satisfied \cite{AGI}.
We do not know any algebra such that (T$'$) is not satisfied.

As a basic tool of the study of silting objects, we shall introduce
a partial order on silting objects (Theorem \ref{partial order}) and establish the relationship with
`silting mutation' (Theorem \ref{two quivers}) by generalizing the theory of Riedtmann-Schofield and
Happel-Unger for tilting modules.
In particular, Question \ref{transitive question} is equivalent to the following question.

\begin{question}
Let $A$ be a finite dimensional algebra over a field.
When is the Hasse quiver of the partially ordered set of basic silting objects in $\K^{\rm b}(\proj{A})$ connected?
\end{question}

%asking whether the Hasse quiver of the partially ordered set $\silt\TT$ is connected or not.

The comparison of three kinds of mutation is explained by the following table.
\[\begin{array}{|c||c|c|c|}
\hline
&\mbox{always possible}&\mbox{partial order exists}&\mbox{transitive for hereditary case}\\ \hline\hline
\mbox{cluster tilting mutation}&\mbox{Yes}&\mbox{No}&\mbox{Yes}\\ \hline
\mbox{silting mutation}&\mbox{Yes}&\mbox{Yes}&\mbox{Yes}\\
\cup&&&\\
\mbox{tilting mutation}&\mbox{No}&\mbox{Yes}&\mbox{No}\\ \hline
\mbox{exceptional mutation}&\mbox{Yes}&\mbox{No}&\mbox{Yes}\\ \hline
\end{array}\]

In our paper we study basic properties of silting objects/subcategories.
We show that non-isomorphic indecomposable objects in any silting subcategory form a basis of the Grothendieck group of the triangulated category
(Theorem \ref{grothendieck}).
In particular all basic silting objects have the same number of indecomposable summands (Corollary \ref{thm:number of indecomposable direct summands}).
We also introduce `silting reduction' (Theorem \ref{silting reduction}), which gives a bijection between silting subcategories containing a certain fixed subcategory
and silting subcategories in a certain quotient triangulated category.

In section \ref{The correspondence between silting subcategories and t-structures}
we study in detail the relationship between silting subcategories and t-structures under the assumption that the triangulated categories have arbitrary coproducts,
and improve some of pioneering results of Hoshino-Kato-Miyachi \cite{HKM}.
The advantage of this setting is that each set of compact objects gives rise to a torsion pair (Theorem \ref{torsion pair}),
which is not the case for the setting of section \ref{section: silting subcategories}.
This basic result seems to be new and to have independent interest.
As an application, we establish a one-to-one correspondence between silting subcategories
and t-structures satisfying certain conditions (Theorem \ref{correspondence}).
%Then we study the hearts of certain t-structures including those associated with silting subcategories.
%We show that the hearts contain generating subcategories consisting of compact projective objects,
%and in particular we realize hearts as certain functor categories (Theorem \ref{category P}, Corollary \ref{realization1}).

We notice that a different generalization of `tilting mutation' was given in \cite{IO,HX}.
Also some aspects of mutation were recently discussed in \cite{BRT,L,KY}.

Parts of results in this paper were presented in Trondheim (March 2009),
Nagoya (June 2009), Matsumoto (October 2009), Shizuoka (November 2009) and Tokyo (August 2010).

\medskip
\noindent
{\bf Notations }
Let $\TT$ be an additive category.
For morphisms $f:X\to Y$ and $g:Y\to Z$ in a category, we denote by $gf:X\to Z$ the composition.
We say that a morphism $f:X\to Y$ is \emph{right minimal} if any morphism $g:X\to X$
satisfying $fg=f$ is an isomorphism. Dually we define a \emph{left minimal} morphism.

For a collection $\XX$ of objects in $\TT$,
we denote by $\add\XX$ (respectively, $\Add\XX$) the smallest full subcategory of $\TT$
which is closed under finite (respectively, arbitrary) coproducts, summands and isomorphisms
and contains $\XX$.
We denote by $\summand\XX$ the smallest full subcategory of $\TT$ which is closed under summands and contains $\XX$.
When we say that $\XX$ is a \emph{subcategory} of $\TT$, we always assume that
$\XX$ is full and satisfies $\XX=\add\XX$.

Let $\XX$ be a subcategory of $\TT$.
We say that a morphism $f:X\to Y$ is a \emph{right $\XX$-approximation} of $Y$ 
if $X\in \XX$ and $\homo{\TT}{X}{f}$ is surjective for any $X\in \XX$. 
We say that $\XX$ is \emph{contravariantly finite} if any object in $\TT$ has a right $\XX$-approximation. 
Dually, we define a \emph{left $\XX$-approximation} and a \emph{covariantly finite subcategory}. 
We say that $\XX$ is \emph{functorially finite} if it is contravariantly and covariantly finite. 

When $\TT$ is a triangulated category, we denote by $\thick\XX$ the smallest thick subcategory of $\TT$ containing $\XX$.
For collections $\XX$ and $\YY$ of objects in $\TT$, we denote by
$\XX*\YY$ the collection of objects $Z\in\TT$ appearing in a triangle
$X\to Z\to Y\to X[1]$ with $X\in\XX$ and $Y\in\YY$.
We set
\[\begin{array}{ccccc}
\XX^\perp   &=& \XX^{\perp_{\TT}} &:=& \{T\in \TT\ |\ \Hom_{\TT}(\XX,T)=0\},\\
{}^\perp\XX &=& {}^{\perp_{\TT}}\XX &:=& \{T\in \TT\ |\ \Hom_{\TT}(T,\XX)=0\}
\end{array}\]

For an additive category $\AA$, we denote by $\K^{\rm b}(\AA)$ the homotopy
category of bounded complexes over $\AA$.
For an abelian category $\AA$, we denote by $\D^{\rm b}(\AA)$ the bounded
derived category of $\AA$.

For a ring $A$, we denote by $\Mod A$ (respectively, $\modu A$) the category
of all (respectively, finitely generated) right $A$-modules,
by $\proj A$ the category of finitely generated projective $A$-modules.
%and by $\fl A$ the category of finite length $A$-modules.
We denote by $J_A$ the Jacobson radical of $A$.
When $A$ is a finite dimensional algebra over a field $k$,
we denote by $D:=\Hom_k(-,k):\modu A\leftrightarrow\modu A^{\rm op}$ the $k$-duality.

\begin{definition}
Let $\TT$ be a triangulated category and $(\XX,\YY)$ a pair of subcategories of $\TT$.
\begin{itemize}
\item We say that $(\XX,\YY)$ is a \emph{torsion pair} 
if $\Hom_{\TT}(\XX,\YY)=0$ and $\XX*\YY=\TT$.
\item We say that $(\XX,\YY)$ is a \emph{t-structure} \cite{BBD}
if $(\XX[1],\YY)$ is a torsion pair and $\XX[1]\subset\XX$.
In this case $\XX\cap\YY$ is called the \emph{heart}.
\item We say that $(\XX,\YY)$ is a \emph{co-t-structure} \cite{P1} 
if $(\XX[-1],\YY)$ is a torsion pair and $\XX\subset\XX[1]$.
In this case $\XX\cap\YY$ is called the \emph{coheart}.
\item We say that a torsion pair $(\XX,\YY)$ is a \emph{stable t-structure} 
\cite{M} if $\XX=\XX[1]$.
\end{itemize}
\end{definition}

\medskip
\noindent
{\bf Acknowledgements }
The second author would like to thank Jiaqun Wei for stimulating discussion.
The authors would like to thank David Pauksztello, Apostolos Beligiannis and Sonia Trepode for useful information.

%%%%%%%%%%%%%%%%%%%%%%%%%%%%%%%%%%%%%%%%%%%%%%%%%%%%%%%%%%%%%%%%%%%%%%%%%%%%%%%%%%%%%%%%%%%%%%%%%%%%%%%%%%%
%%%%%%%%%%%%%%%%%%%%%%%%%%%%%%%%%%%%%%%%%%%%%%%%%%%%%%%%%%%%%%%%%%%%%%%%%%%%%%%%%%%%%%%%%%%%%%%%%%%%%%%%%%%

\section{Silting subcategories}\label{section: silting subcategories}

Let $\TT$ be a triangulated category. We do not assume anything else on $\TT$ unless otherwise stated.

Throughout this paper we write the vanishing condition $\Hom_{\TT}(X,Y[i])=0$ for any $i\in\Z$ by
\[\Hom_{\TT}(X,Y[\Z])=0\]
simply.  Similarly we often use
\[\Hom_{\TT}(X,Y[>0])=0,\ \Hom_{\TT}(X,Y[<0])=0,\ \Hom_{\TT}(X,Y[\neq0])=0\ \mbox{ and so on.}\]

\subsection{Definition and basic properties}

In this subsection we introduce silting subcategories/objects of triangulated categories
and study their basic properties.

\begin{definition}\label{definition of silting}
Let $\MM$ be a subcategory of a triangulated category $\TT$.
\begin{itemize}
\item[(a)] We say that $\MM$ is \emph{silting} if $\Hom_{\TT}(\MM, \MM[>0])=0$ and $\TT=\thick\MM$. 
We denote by $\silt\TT$ the collection of silting subcategories in $\TT$.
\item[(b)] We say that $\MM$ is \emph{tilting} if $\Hom_{\TT}(\MM, \MM[\neq0])=0$ and $\TT=\thick\MM$.
\item[(c)] We say that an object $M\in\TT$ is \emph{silting} (respectively, \emph{tilting}) if so is $\add M$.
\end{itemize}
\end{definition}

The following examples give typical classes of tilting/silting objects.

\begin{example}
\begin{itemize}
\item[(a)] Let $A$ be a ring.
Then $A$ regarded as a stalk complex is a tilting object in $\K^{\rm b}(\proj A)$.
More generally, any tilting complex of $A$ is a tilting object in $\K^{\rm b}(\proj A)$.
\item[(b)] Let $A$ be a differential graded ring and let $\D(A)$ be the derived category of $A$.
If $H^i(A)=0$ for any $i>0$, then the thick subcategory $\thick(A)$ of $\D(A)$ generated by $A$
has a silting object $A$.
\end{itemize}
\end{example}

The above example (a) is standard among triangulated categories with tilting objects in the following sense:

\begin{proposition}\cite{K}
Let $\TT$ be an algebraic triangulated category.
If $\TT$ has a tilting object $M$, then $\TT$ is triangle equivalent
to $\K^{\rm b}(\proj\End_{\TT}(M))$.
\end{proposition}

%No similar characterization of (algebraic) triangulated categories containing silting objects is known.

On the other hand, we have the following necessary conditions for existence of silting/tilting subcategories.

\begin{proposition}\label{silting vanishing}
Let $\TT$ be a triangulated category.
If $\TT$ contains a silting (respectively, tilting) subcategory,
then for any $X,Y\in\TT$ we have $\Hom_{\TT}(X,Y[i])=0$ for $i\gg0$ (respectively, $|i|\gg0$).
\end{proposition}

\begin{proof}
We only show the statement for silting subcategories.
Let $\UU:=\{X\in\TT\ |\ \Hom_{\TT}(X,M[\gg0])=0$ for any $M\in\MM\}$.
Then $\UU$ is a thick subcategory of $\TT$ containing $\MM$.
Thus we have $\UU=\TT$ since $\thick\MM=\TT$.
For any $X\in\TT$, let $\VV_X:=\{Y\in\TT\ |\ \Hom_{\TT}(X,Y[\gg0])=0\}$.
Then we have $\MM\subset\VV_X$ by the above argument.
Since $\VV_X$ is a thick subcategory of $\TT$ containing $\MM$, we have $\VV_X=\TT$ again.
Thus the assertion holds.
\end{proof}

Immediately we have the following observation.

\begin{example}
Let $A$ be a finite dimensional algebra over a field $k$.
\begin{itemize}
\item[(a)] $\D^{\rm b}(\modu A)$ contains silting subcategories if and only if ${\rm gl.dim} A<\infty$.
\item[(b)] Assume that $A$ is self-injective.
Then the stable module category $\sta{A}$ of $\modu{A}$ contains silting subcategories if and only if $A$ is semisimple.
\end{itemize}
\end{example}

\begin{proof}
We only have to show `only if' part.
Both for the cases (a) and (b), we have $\Ext_A^i(A/J_A,A/J_A)\neq0$ for any $i>0$.
Thus the assertions follow from Proposition \ref{silting vanishing}.
\end{proof}

It is often the case that triangulated categories have many silting subcategories which are not tilting.
But there are important triangulated categories such that all silting subcategories are tilting.

%the fact that $\K^{\rm{b}}(\proj{A})$ is a $0$-Calabi-Yau triangulated category in the following sense:

\begin{definition}
Let $\ell$ be an integer. We call a $k$-linear $\Hom$-finite triangulated category $\TT$ \emph{$\ell$-Calabi-Yau}
if there is a bifunctorial isomorphism $\homo{\TT}{-}{?}\simeq D\homo{\TT}{?}{-[\ell]}$. 
\end{definition}

%For example if $A$ is a symmetric algebra, then 
%and the stable module category $\sta{A}$ is $(-1)$-Calabi-Yau. 
One can easily check the following statements.

\begin{lemma}\label{prop:CY}
Let $\ell$ be an integer and $\TT$ an $\ell$-Calabi-Yau triangulated category.
If $\TT\neq0$, then the following assertions hold:
\begin{enumerate}[$(1)$]
\item If $\ell=0$, then any silting subcategory of $\TT$ is tilting; 
\item If $\ell>0$, then there exist no silting subcategories of $\TT$;
\item If $\ell<0$, then there exist no tilting subcategories of $\TT$.
\end{enumerate}
\end{lemma}

\begin{proof}
Since any non-zero object $X\in\TT$ satisfies $\homo{\TT}{X}{X[\ell]}\simeq D\homo{\TT}{X}{X}\neq0$,
the assertions (2) and (3) follow. We can prove (1) similarly.
\end{proof}

Immediately we have the following observation.

\begin{example}
Let $A$ be a finite dimensional symmetric algebra over a field.
Then any silting object in $\K^{\rm b}(\proj A)$ is tilting.
\end{example}

\begin{proof}
We know that $\K^{\rm b}(\proj A)$ is a 0-Calabi-Yau triangulated category by Auslander-Reiten duality
\cite{H}. Thus the assertion follows from Lemma \ref{prop:CY}.
\end{proof}

Notice that we do not know whether there exist $\ell$-Calabi-Yau triangulated categories
containing a silting subcategory and $\ell<0$.

We end this subsection with the following remark, where we say that a category
is \emph{skeletally small} if isomorphism classes of objects form a set.

\begin{remark}\label{T is small}
If a triangulated category $\TT$ has a skeletally small silting subcategory,
then $\TT$ is also skeletally small (e.g. Proposition \ref{generate2}).
\end{remark}

%%%%%%%%%%%%%%%%%%%%%%%%%%%%%%%%%%%%%%%%%%%%%%%%%%%%%%%%%%%%%%%%%%%%%%%%%%%%%%%%%%%%%%%%%%%%%%%%%%%%%%%%%%%
%%%%%%%%%%%%%%%%%%%%%%%%%%%%%%%%%%%%%%%%%%%%%%%%%%%%%%%%%%%%%%%%%%%%%%%%%%%%%%%%%%%%%%%%%%%%%%%%%%%%%%%%%%%

\subsection{Partial order on silting subcategories}

The aim of this subsection is to introduce a partial order on silting subcategories
as a generalization of the partial order on tilting modules introduced by
Riedtmann-Schofield \cite{RS} and Happel-Unger \cite{HU}.
Our main result in this subsection is Theorem \ref{partial order} below.

\begin{definition}
For $\MM,\NN\in\silt\TT$, we write $\MM\ge\NN$ if $\Hom_{\TT}(\MM, \NN[>0])=0$.
\end{definition}

\begin{theorem}\label{partial order}
$\ge$ gives a partial order on $\silt\TT$.
\end{theorem}

The following subcategory plays an important role in the proof of Theorem \ref{partial order}.

\begin{definition}
For any $\MM\in\silt\TT$, we define a subcategory of $\TT$ by
\[\TT_{\MM}^{\le0}:=\{X\in\TT\ |\ \Hom_{\TT}(\MM, X[>0])=0\}.\]
We put $\TT_{\MM}^{\le\ell}=\TT_{\MM}^{<\ell+1}:=\TT_{\MM}^{\le0}[-\ell]$ for $\ell\in\Z$.
\end{definition}

Clearly we have $\MM\subset\TT_{\MM}^{\le0}$ and $\TT_{\MM}^{<0}\subset\TT_{\MM}^{\le0}$.
For any non-zero object $X\in\TT$, we have $X[\gg0]\in\TT_{\MM}^{\le0}$ and $X[\ll0]\notin\TT_{\MM}^{\le0}$ 
by Proposition \ref{silting vanishing}.

Crucial results are the following.

\begin{proposition}\label{recover}
For any $\MM\in\silt\TT$, we have $\TT_{\MM}^{\le0}\cap{}^\perp(\TT_{\MM}^{<0})=\MM$.
\end{proposition}

\begin{proposition}\label{another condition}
Let $\MM,\NN\in\silt\TT$.
Then $\MM\ge\NN$ if and only if $\TT_{\MM}^{\le0}\supset\TT_{\NN}^{\le0}$.
\end{proposition}

Before proving these Propositions, we prove Theorem \ref{partial order} by using them.

(i) We shall show that $\MM\ge\NN$ and $\NN\ge\LL$ imply $\MM\ge\LL$.
By Proposition \ref{another condition}, we have $\TT_{\MM}^{\le0}\supset\TT_{\NN}^{\le0}\supset\TT_{\LL}^{\le0}$.
Thus $\TT_{\MM}^{\le0}\supset\TT_{\LL}^{\le0}$, and we have $\MM\ge\LL$ again by Proposition \ref{another condition}.

(ii) We shall show that $\MM\ge\NN$ and $\NN\ge\MM$ imply $\MM=\NN$.
By Proposition \ref{another condition}, we have $\TT_{\MM}^{\le0}=\TT_{\NN}^{\le0}$.
Thus we have ${}^\perp(\TT_{\MM}^{<0})={}^\perp(\TT_{\NN}^{<0})$.
By Proposition \ref{recover}, we have $\MM=\NN$.
\qed

\medskip
To prove Propositions \ref{recover} and \ref{another condition}, we need the following result.

\begin{lemma}\label{generate}
\begin{itemize}
\item[(a)] For any subcategory $\MM$ of $\TT$, we have
\[\thick\MM=\bigcup_{\ell>0,\ n_1,\cdots,n_\ell\in\Z}\summand
(\MM [n_1]*\cdots*\MM [n_{\ell}]).\]
\item[(b)] Moreover if $\Hom_{\TT}(\MM,\MM[>0])=0$, then we have
\[\thick\MM=\bigcup_{\ell\ge0}
\summand(\MM[-\ell]*\MM[1-\ell]*\cdots*\MM[\ell-1]*\MM[\ell]).\]
\end{itemize}
\end{lemma}

\begin{proof}
(a) Clearly the right hand side is contained in $\thick\MM$.
Since the right hand side clearly forms a thick subcategory of $\TT$, it contains $\thick\MM$.

(b) If $n\ge m$, then any triangle
\[M[n]\to X\to M'[m]\xrightarrow{f}M[n+1]\]
with $M,M'\in\MM$ satisfies $f=0$. This implies $X\simeq M[n]\oplus M'[m]$, so we have
$\MM[n]*\MM[m]\subset\MM[m]*\MM[n]$.
Thus we have reduced to the assertion from (a).
\end{proof}

The following observation is clear since we only consider subcategories $\MM$ satisfying $\MM=\add\MM$ by our notations.

\begin{lemma}\label{reduction}
Let $\MM$ and $\NN$ be subcategories of $\TT$ and $X\in\summand(\MM*\NN)$.
\begin{itemize}
\item[(a)] If $\Hom_{\TT}(\MM,X)=0$, then $X\in\NN$.
\item[(b)] If $\Hom_{\TT}(X,\NN)=0$, then $X\in\MM$.
\end{itemize}
\end{lemma}

Now we have the following description of $\TT_{\MM}^{\le0}$.

\begin{proposition}\label{generate2}
For any $\MM\in\silt\TT$, we have
\begin{eqnarray*}
\TT&=&\bigcup_{\ell\ge0}\summand(\MM[-\ell]*\MM[1-\ell]*\cdots*\MM[\ell-1]*\MM[\ell]),\\
\TT_{\MM}^{\le0}&=&\bigcup_{\ell\ge0}\summand(\MM*\MM[1]*\cdots*\MM[\ell]).
\end{eqnarray*}
%In particular we have $\TT=\bigcup_{i\ge0}\TT_{\MM}^{\le0}[-i]$.
\end{proposition}

\begin{proof}
The first equation is clear from Lemma \ref{generate}(b).

Fix any $X\in\TT_{\MM}^{\le0}$. We can take the smallest integer $n$ such that
\[X\in\summand(\MM[n]*\MM[n+1]\cdots*\MM[\ell])\]
for some $\ell\ge n$. By Lemma \ref{reduction}, the minimality of $n$ implies
$\Hom_{\TT}(\MM[n],X)\neq0$.
Since $X\in\TT_{\MM}^{\le0}$, we have $n\ge0$.
Thus we have $X\in\summand(\MM*\MM[1]*\cdots*\MM[\ell])$.
\end{proof}

Now we have the following important property of silting subcategories.

\begin{theorem}\label{no inclusion}
If $\MM,\NN\in\silt\TT$ satisfy $\MM\subset\NN$, then $\MM=\NN$.
\end{theorem}

\begin{proof}
Fix $X\in\NN$. Since $\NN$ is silting, we have $X\in\TT_{\MM}^{\le0}$.
By Proposition \ref{generate2}, we have
\[X\in\summand(\MM*\MM[1]*\cdots*\MM[\ell])\]
for some $\ell\ge0$.
Since $\Hom_{\TT}(X,\MM[1]*\cdots*\MM[\ell])=0$, we have $X\in\MM$
by Lemma \ref{reduction}.
\end{proof}

Now we give a proof of Proposition \ref{recover}.

Put $\NN:=\TT_{\MM}^{\le0}\cap{}^\perp(\TT_{\MM}^{<0})$.
We have $\MM\subset\TT_{\MM}^{\le0}$.
Since we have $\Hom_{\TT}(\MM,\TT_{\MM}^{<0})=0$, we have
$\MM\subset{}^\perp(\TT_{\MM}^{<0})$.
Thus we have $\MM\subset\NN$.

On the other hand, we have
\[\Hom_{\TT}(\NN,\NN[>0])\subset\Hom_{\TT}({}^\perp(\TT_{\MM}^{<0}),\TT_{\MM}^{\le0}[>0])=0.\]
Consequently, $\NN$ is also a silting subcategory of $\TT$.
By Theorem \ref{no inclusion}, we have $\MM=\NN$.
\qed

\medskip
Now we give a proof of Proposition \ref{another condition}.

First we assume $\TT_{\MM}^{\le0}\supset\TT_{\NN}^{\le0}$. Then we have $\NN\subset\TT_{\MM}^{\le0}$,
which implies $\Hom_{\TT}(\MM,\NN[>0])=0$.

Conversely we assume $\MM\ge\NN$. Then we have $\NN[\ge0]\subset\TT_{\MM}^{\le0}$.
By Proposition \ref{generate2}, we have
\[\TT_{\NN}^{\le0}=\bigcup_{\ell\ge0}\summand(\NN*\NN[1]*\cdots*\NN[\ell])
\subset\TT_{\MM}^{\le0}.\]
Thus we have completed the proof.
\qed

\medskip
Now we give the following property of the partial order.

\begin{proposition}\label{common summand}
If $\MM,\NN,\LL\in\silt\TT$ satisfy
$\MM\ge\LL\ge\NN$, then $\MM\cap\NN\subset\LL$.
\end{proposition}

\begin{proof}
By Proposition \ref{another condition}, we have $\TT_{\MM}^{\le0}\supset\TT_{\LL}^{\le0}\supset\TT_{\NN}^{\le0}$ and
${}^\perp(\TT_{\MM}^{<0})\subset{}^\perp(\TT_{\LL}^{<0})\subset{}^\perp(\TT_{\NN}^{<0})$.
By Proposition \ref{recover}, we have
\[\MM\cap\NN\subset{}^\perp(\TT_{\MM}^{\le0})\cap\TT_{\NN}^{\le0}
\subset{}^\perp(\TT_{\LL}^{\le0})\cap\TT_{\LL}^{\le0}=\LL.\]
\end{proof}

We end this subsection by the following observation.

\begin{proposition}\label{existence of additive generator}
If $\TT$ has a silting object, then any silting subcategory contains an additive generator.
\end{proposition}

\begin{proof}
Let $M$ be a silting object and $\NN$ be a silting subcategory.
By Lemma \ref{generate}, there exists $N_1,\ldots,N_\ell\in\NN$ and $n_1,\cdots,n_\ell\in\Z$ such that $M\in\summand(N_1[n_1]*\cdots*N_\ell[n_\ell])$.
Let $\NN':=\add(N_1\oplus\cdots\oplus N_\ell)$.
Since $M\in\thick\NN'$, we have $\TT=\thick\NN'$.
Thus $\NN'$ is a silting subcategory of $\TT$.
By Theorem \ref{no inclusion}, we have $\NN=\NN'$.
Thus $\NN$ has an additive generator.
\end{proof}

%%%%%%%%%%%%%%%%%%%%%%%%%%%%%%%%%%%%%%%%%%%%%%%%%%%%%%%%%%%%%%%%%%%%%%%%%%%%%%%%%%%%%%%%%%%%%%%%%%%%%%%%%%%
%%%%%%%%%%%%%%%%%%%%%%%%%%%%%%%%%%%%%%%%%%%%%%%%%%%%%%%%%%%%%%%%%%%%%%%%%%%%%%%%%%%%%%%%%%%%%%%%%%%%%%%%%%%

\subsection{Krull-Schmidt triangulated categories}

Let $\TT$ be a triangulated category.
In this subsection we always assume that $\TT$ is \emph{Krull-Schmidt}
in the sense that any object in $\TT$ is isomorphic to a finite coproduct
of objects whose endomorphism rings are local.
In this case such a coproduct is uniquely determined up to isomorphism.
We denote by $J_{\TT}$ the \emph{Jacobson radical} of $\TT$ \cite{ARS,ASS}.
For a subcategory $\MM$ of $\TT$, we denote by $\ind\MM$ the set
of isoclasses of indecomposable objects in $\MM$.

We say that an object $M\in\TT$ is \emph{basic} if $M$ is isomorphic to
a coproduct of indecomposable objects which are mutually non-isomorphic.
Since $\TT$ is Krull-Schmidt, we have a one-to-one correspondence
between the isomorphism classes of basic objects $M$ and subcategories $\MM$ of $\TT$
containing additive generators. It is given by $M\mapsto\MM=\add M$.

\begin{proposition}
Assume that $\TT$ has a silting object.
Then we can regard $\silt\TT$ as the set of isomorphism classes of basic silting objects in $\TT$.
\end{proposition}

\begin{proof}
$\silt\TT$ is a set by Remark \ref{T is small}.
By Proposition \ref{existence of additive generator}, any silting subcategory of $\TT$ is an additive closure of a silting object.
Thus we have the assertion.
\end{proof}

Thanks to Krull-Schmidt assumption, we have the following useful property (e.g. \cite[2.1, 2.3]{IY}),
where (b) and (c) is famous as a `Wakamatsu's Lemma'.

\begin{lemma}\label{direct summand}
Let $\MM$ be a subcategory of $\TT$.
Then the following statements hold.
\begin{itemize}
\item[(a)] If a subcategory $\NN$ of $\TT$ satisfies $\Hom_{\TT}(\MM,\NN)=0$, 
then $\MM*\NN$ is closed under summands.
\item[(b)] If $\MM$ is contravariantly finite and $\MM*\MM\subset\MM$, 
then $(\MM, \MM^{\perp})$ is a torsion pair.
\item[(c)] If $\MM$ is covariantly finite and $\MM*\MM\subset\MM$, 
then $({}^{\perp}\MM, \MM)$ is a torsion pair. 
\end{itemize}
\end{lemma}

Now we have the following equalities.

\begin{proposition}\label{generate3}
\begin{itemize}
\item[(a)] For any $\MM\in\silt\TT$, we have
\begin{eqnarray*}%\label{T is M}
\TT&=&\bigcup_{\ell\ge0}\MM[-\ell]*\MM[1-\ell]*\cdots*\MM[\ell-1]*\MM[\ell],\\ %\label{T< is M}
\TT_{\MM}^{\le0}&=&\bigcup_{\ell\ge0}\MM*\MM[1]*\cdots*\MM[\ell],\\ 
{}^{\perp}(\TT^{\leq 0}_{\MM})&=&\bigcup_{\ell>0}\MM[-\ell]*\MM[-\ell+1]*\cdots*\MM[-1].
\end{eqnarray*}
\item[(b)] $({}^{\perp}(\TT^{\leq 0}_{\MM}), \TT^{\leq 0}_{\MM})$ is a torsion pair, and so $({}^{\perp}(\TT^{<0}_{\MM}), \TT^{\leq 0}_{\MM})$ is a co-t-structure with the coheart $\MM$.
\end{itemize}
\end{proposition}

\begin{proof}
The first two equalities follow from Proposition \ref{generate2} and Lemma \ref{direct summand}(a).
By the second equality we have 
\[{}^{\perp}(\TT^{\leq 0}_{\MM})\supset\bigcup_{\ell>0}\MM[-\ell]*\MM[-\ell+1]*\cdots*\MM[-1].\]
Together with the first equality we have
\[(\bigcup_{\ell>0}\MM[-\ell]*\cdots*\MM[-1])*\TT^{\leq 0}_{\MM}=(\bigcup_{\ell>0}\MM[-\ell]*\cdots*\MM[-1])*(\bigcup_{\ell>0}\MM*\cdots\MM[\ell])=\TT.\]
Thus $(\bigcup_{\ell>0}\MM[-\ell]*\cdots*\MM[-1], \TT^{\leq 0}_{\MM})$ is a torsion pair, and hence we have
\[{}^{\perp}(\TT^{\leq 0}_{\MM})\subset\bigcup_{\ell>0}\MM[-\ell]*\MM[-\ell+1]*\cdots*\MM[-1].\]
Thus we get the third equality, and $({}^{\perp}(\TT^{\leq 0}_{\MM}), \TT^{\leq 0}_{\MM})$ is a torsion pair.
The coheart of the co-t-structure $({}^{\perp}(\TT^{<0}_{\MM}), \TT^{\leq 0}_{\MM})$ is $\MM$ by Proposition \ref{recover}.
\end{proof}

\begin{proposition}\label{resolution}
Let $\MM\in\silt\TT$. For any $N=N_0\in\TT_{\MM}^{\le0}$, we have triangles
\[\xymatrix@R=0.2cm{
N_1\ar[r]^{g_1}&M_0\ar[r]^{f_0}&N_0\ar[r]&N_1[1],\\
&\cdots,\\
N_\ell\ar[r]^{g_\ell}&M_{\ell-1}\ar[r]^{f_{\ell-1}}&N_{\ell-1}\ar[r]&N_\ell[1],\\
0\ar[r]^{g_{\ell+1}}&M_\ell\ar[r]^{f_\ell}&N_\ell\ar[r]&0,
}\]
for some $\ell\ge0$ such that $f_i$ is a minimal right $\MM$-approximation and
$g_{i+1}$ belongs to $J_{\TT}$ for any $0\le i\le \ell$.
\end{proposition}

\begin{proof}
Since $N_0\in\TT_{\MM}^{\le0}$, we have
\[N_0\in\MM*\MM[1]*\cdots*\MM[\ell]\]
for some $\ell\ge0$ by Proposition \ref{generate3}.
We can assume $\ell>0$. Then we have a triangle
\[N'_1\to M'_0\xrightarrow{f'_0}N_0\to N'_1[1]\]
with $M'_0\in\MM$ and $N'_1\in\MM*\MM[1]*\cdots*\MM[\ell-1]$.
Since $N'_1\in\TT_{\MM}^{\le0}$, we have that $f'_0$ is a right $\MM$-approximation.
Thus we can write
$f'_0=(f_0\ 0):M_0=M'_0\oplus M''_0\to N_0$ with a minimal right $\MM$-approximation $f_0$.
Then we have a triangle
\[N_1\xrightarrow{g_1}M_0\xrightarrow{f_0}N_0\to N_1[1]\]
such that $g_1$ belongs to $J_{\TT}$ and $N_1$ is a summand of $N'_1$.
By Lemma \ref{direct summand}, we have $N_1\in\MM*\MM[1]*\cdots*\MM[\ell-1]$.
Repeating similar construction, we obtain the desired triangles.
\end{proof}

The following observation is often useful.

\begin{lemma}\label{no common summand}
Let $\MM\in\silt\TT$ and $N_0,N_0'\in\TT_{\MM}^{\le0}$.
For $N_0$, we take $\ell\ge0$ and triangles in Proposition \ref{resolution}.
Also for $N_0'$, we take triangles 
\[\xymatrix@R=0.2cm{
N_1'\ar[r]^{g_1'}&M_0'\ar[r]^{f_0'}&N_0'\ar[r]&N_1'[1],\\
&\cdots,\\
N_{\ell'}'\ar[r]^{g_\ell'}&M_{\ell'-1}'\ar[r]^{f_{\ell'-1}'}&N_{\ell'-1}'\ar[r]&N_{\ell'}'[1],\\
0\ar[r]^{g_{\ell'+1}'}&M_{\ell'}'\ar[r]^{f_{\ell'}'}&N_{\ell'}'\ar[r]&0,
}\]
satisfying the same properties.
If $\Hom_{\TT}(N_0,N_0'[\ell])=0$ holds, then we have $(\add M_\ell)\cap(\add M_0')=0$.
\end{lemma}

\begin{proof}
For $\ell=0$, the assertion follows from right minimality of $f_0'$. So we assume $\ell>0$.
We only have to show that any morphism $a:M_\ell\to M_0'$ belongs
to $J_{\TT}$. Applying $\Hom_{\TT}(-,N_0')$ to triangles in Proposition \ref{resolution}, we have
\[\Hom_{\TT}(N_{\ell-1}[-1],N_0')\simeq\Hom_{\TT}(N_{\ell-2}[-2],N_0')
\simeq\cdots\simeq\Hom_{\TT}(N_0[-\ell],N_0')=0.\]
Thus we have the following commutative diagram of triangles.
\[\xymatrix{
N_{\ell-1}[-1] \ar[r] \ar[d] & M_\ell \ar[r]^{g_{\ell}} \ar[d]^{a} & M_{\ell-1} \ar[r]^{f_{\ell-1}} \ar[d]^{b} & N_{\ell-1}\ar[d] \\
N_1' \ar[r]^{g_{1}'} & M_0' \ar[r]^{f_{0}'} & N_0'\ar[r]&N_1'[1].
}\]
Since $f_0'$ is a right $\MM$-approximation,
there exists $b':M_{\ell-1}\to M_0'$ such that $b=f_0'b'$.
Since $f_0'(a-b'g_\ell)=0$, there exists $a':M_\ell\to N_1'$
such that $a=b'g_\ell+g_1'a'$. Since both $g_\ell$ and $g_1'$ belong to $J_{\TT}$,
we have $a\in J_{\TT}$.
Thus the proof is completed.
\end{proof}

As an application, we have the following result.

\begin{theorem}\label{indecomposable case}
If $\TT$ has an indecomposable silting object $M$, 
then we have $\silt{\TT}=\{M[i]\ |\ i\in \Bbb{Z}\}$.  
\end{theorem}

\begin{proof}
Let $N$ be a basic silting object in $\TT$. 
Take the smallest integer $k\in\Z$ such that $\homo{\TT}{M}{N[k]}\not=0$.
Replacing $N$ by $N[k]$, we can assume $N\in \TT^{\leq 0}_{M}$
and $\homo{\TT}{M}{N}\not=0$. 
We have triangles in Proposition \ref{resolution}. 
Then we have $M_0\neq0$ since $\homo{\TT}{M}{N}\neq0$,
and moreover we can assume $M_{\ell}\neq0$.
%\[\xymatrix@!R=1pt{
%N_{1} \ar[r] & M_{0} \ar[r]^{f_{0}} & N \ar[r] & N_{1}[1] \\
%\cdots ,\\
%N_{\ell} \ar[r] & M_{\ell-1} \ar[r]^{f_{\ell-1}} & N_{\ell-1} \ar[r] & N_{\ell}[1] \\
%0 \ar[r] & M_{\ell} \ar[r]^{f_{\ell}} & N_{\ell} \ar[r] & 0
%}\]
%for some $\ell\geq 0$ such that $f_{i}$ is a minimal right $\add{M}$-approximation. 

Assume $\ell>0$. Then we have $\Hom_{\TT}(N,N[\ell])=0$ and so
$\add{M_{\ell}}\cap \add{M_{0}}=0$ by Lemma \ref{no common summand}. 
Since both $M_0$ and $M_{\ell}$ are non-zero objects in $\add M$
where $M$ is indecomposable, this is a contradiction.
Thus we have $\ell=0$ and $N\simeq M_0\in \add M$.
Since $N$ is basic, we have $N\simeq M$.
\end{proof}

%%%%%%%%%%%%%%%%%%%%%%%%%%%%%%%%%%%%%%%%%%%%%%%%%%%%%%%%%%%%%%%%%%%%%%%%%%%%%%%%%%%%%%%%%%%%%%%%%%%%%%%%%%%

Next we shall show the following description of Grothendieck groups of triangulated categories with silting subcategories.

\begin{theorem}\label{grothendieck}
Let $\TT$ be a Krull-Schmidt triangulated category with a silting subcategory $\MM$.
Then the Grothendieck group $K_0(\TT)$ of $\TT$ is a free abelian group with a basis $\ind\MM$.
\end{theorem}
 
For an object $X\in\TT$, we denote by $\delta(X)$ the number of non-isomorphic indecomposable summands of $X$.
As an immediate consequence of Theorem \ref{grothendieck}, we have the following result.

\begin{corollary}\label{thm:number of indecomposable direct summands}
For any silting objects $M,N\in \TT$, we have $\delta (M)=\delta (N)$.
\end{corollary}

Let us start with proving Theorem \ref{grothendieck}.

Since $\thick\MM=\TT$, we have that $\ind\MM$ generate $K_0(\TT)$.
It is enough to show that they are linearly independent.
To prove this, we shall define a map
\[\gamma:{\rm ob}\TT\to\Z^{\ind\MM}.\]
\begin{itemize}
\item The map $\gamma$ naturally identifies ${\rm ob}\MM$ with $\Z_{\ge0}^{\ind\MM}$.
%\[\gamma(M^{(1)}{}^{a_1}\coprod\cdots\coprod M^{(n)}{}^{a_n}):=(a_1,\cdots,a_n).\]
\item For any $N\in\TT_{\MM}^{\le0}$, we take triangles in Proposition \ref{resolution} and put
\[\gamma(N):=\sum_{i=0}^\ell(-1)^i\gamma(M_i).\]
\item For general $N\in\TT$, take a sufficiently large $k$ such that $N[k]\in\TT_{\MM}^{\le0}$ and put
\[\gamma(N):=(-1)^k\gamma(N[k]).\]
\end{itemize}
In other words, for any $\ell\ge0$ and $M_i\in\MM$, we put
\[\gamma(M_{-\ell}[-\ell]*M_{1-\ell}[-\ell]*\cdots*M_{\ell-1}[\ell-1]*M_\ell[\ell]):=\sum_{i=-\ell}^\ell(-1)^i\gamma(M_i).\]
The crucial step is to prove the following observation.

\begin{lemma}\label{gamma is on grothendieck}
Let $X\to Y\to Z\to X[1]$ be a triangle in $\TT$. Then we have $\gamma(X)-\gamma(Y)+\gamma(Z)=0$.
\end{lemma}

\begin{proof}
Without loss of generality, we can assume $X,Y,Z\in\TT_{\MM}^{\le0}$.
To use induction, we consider the following assertions for $\ell\ge0$.

(i)$_\ell$ The statement is true if $X,Y,Z\in \MM*\MM[1]*\cdots*\MM[\ell]$.

(ii)$_\ell$ The statement is true if $X,Y\in \MM*\MM[1]*\cdots*\MM[\ell]$ and $Z\in\MM*\MM[1]*\cdots*\MM[\ell+1]$.

(iii)$_\ell$ The statement is true if $X\in \MM*\MM[1]*\cdots*\MM[\ell]$ and $Y,Z\in\MM*\MM[1]*\cdots*\MM[\ell+1]$.

The statement (i)$_0$ is true since any triangle $X\to Y\to Z\to X[1]$ with $X,Y,Z\in\MM$ splits,
so we have $Y\simeq X\coprod Z$ and $\gamma(X)-\gamma(Y)+\gamma(Z)=0$.

We shall show (i)$_\ell\Rightarrow$(ii)$_\ell$ for any $\ell\ge0$.

Let $X\to Y\to Z\to X[1]$ be a triangle with $X,Y\in \MM*\MM[1]*\cdots*\MM[\ell]$ and $Z\in\MM*\MM[1]*\cdots*\MM[\ell+1]$.
Take a triangle $Z_1\to M_0\to Z\to Z_1[1]$ such that $M_0\in\MM$, $Z_1\in\MM*\MM[1]*\cdots\MM[\ell]$ and
\begin{equation}\label{g1}
\gamma(Z)=\gamma(M_0)-\gamma(Z_1).
\end{equation}
By octahedral axiom, we have the following commutative diagram of triangles:
\[\xymatrix@R=.4cm{
&Z_1[1]\ar@{=}[r]&Z_1[1]\\
X\ar[r]&Y\ar[r]\ar[u]&Z\ar[r]\ar[u]&X[1]\\
X\ar[r]\ar@{=}[u]&W\ar[r]\ar[u]&M_0\ar[r]\ar[u]&X[1]\ar@{=}[u]\\
&Z_1\ar@{=}[r]\ar[u]&Z_1\ar[u].
}\]
Since $\Hom_{\TT}(M_0,X[1])=0$, the lower horizontal triangle splits and we have $W\simeq X\oplus M_0$. In particular, we have
\begin{equation}\label{g2}
\gamma(W)=\gamma(X)+\gamma(M_0).
\end{equation}
Since the left vertical triangle $Z_1\to W\to Y\to Z_1[1]$
satisfies $Z_1,W,Y\in\MM*\MM[1]*\cdots*\MM[\ell]$,
our assumption (i)$_\ell$ implies
\begin{equation}\label{g3}
\gamma(Z_1)-\gamma(W)+\gamma(Y)=0.
\end{equation}
Using \eqref{g1}, \eqref{g2} and \eqref{g3}, we have
\[\gamma(X)-\gamma(Y)+\gamma(Z)=
(\gamma(Z)-\gamma(M_0)+\gamma(Z_1))+(\gamma(X)-\gamma(W)+\gamma(M_0))-(\gamma(Z_1)-\gamma(W)+\gamma(Y))=0.\]
Thus (ii)$_\ell$ holds.

By a quite similar argument, one can show (ii)$_\ell\Rightarrow$(iii)$_\ell\Rightarrow$(i)$_{\ell+1}$ for any $\ell\ge0$.
Thus the assertion follows inductively.
\end{proof}

Now we are ready to prove Theorem \ref{grothendieck}.

By Lemma \ref{gamma is on grothendieck}, we have a homomorphism $\gamma:K_0(\TT)\to\Z^{\ind\MM}$.
Since $\gamma(\ind\MM)$ is a basis of $\Z^n$,
the set $\ind\MM$ must be linearly independent in $K_0(\TT)$.
Thus it forms a basis of $K_0(\TT)$.
\qed

%%%%%%%%%%%%%%%%%%%%%%%%%%%%%%%%%%%%%%%%%%%%%%%%%%%%%%%%%%%%%%%%%%%%%%%%%%%%%%%%%%%%%%%%%%%%%%%%%%%%%%%%%%%
%%%%%%%%%%%%%%%%%%%%%%%%%%%%%%%%%%%%%%%%%%%%%%%%%%%%%%%%%%%%%%%%%%%%%%%%%%%%%%%%%%%%%%%%%%%%%%%%%%%%%%%%%%%

\subsection{Silting mutation}

The aim of this subsection is to introduce silting mutation and give its basic properties.
Let $\TT$ be a triangulated category. We do not assume anything else on $\TT$ unless otherwise stated.

\begin{definition}\label{mutation}
Let $\MM\in\silt\TT$. For a covariantly finite subcategory $\DD$ of $\MM$, we define a subcategory $\mu^+(\MM;\DD)$ of $\TT$ as follows:
For any $M\in\MM$ we take a left $\DD$-approximation $f:M\to D$ and a triangle
\begin{equation}\label{exchange}
M\xrightarrow{f}D\xrightarrow{g}N_M\to M[1].
\end{equation}
We put
\begin{eqnarray*}
\mu^+(\MM;\DD)&:=&\add(\DD\cup\{N_M\ |\ M\in\MM\}).
\end{eqnarray*}
It is easily checked that $\mu^+(\MM;\DD)$ does not depend on a choice of left approximation $f$.
We call $\mu^+(\MM;\DD)$ a \emph{left mutation} of $\MM$.
Dually, we define a \emph{right mutation} $\mu^-(\MM;\DD)$ for a contravariantly finite subcategory $\DD$ of $\MM$.
\emph{(Silting) mutation} is a left or right mutation.
\end{definition}

\begin{theorem}\label{thm:mutation is silting}
Any mutation of a silting subcategory is again a silting subcategory.
\end{theorem}

\begin{proof}
The triangle \eqref{exchange} shows $\thick\mu^+(\MM;\DD)\supset\thick\MM=\TT$.

We only have to show $\Hom_{\TT}(\mu^+(\MM;\DD),\mu^+(\MM;\DD)[>0])=0$.

Applying $\Hom_{\TT}(\DD,-[>0])$ to \eqref{exchange}, we have an exact sequence
\begin{eqnarray*}%\label{apply (M,-)}
0=\Hom_{\TT}(\DD,D[>0])\to\Hom_{\TT}(\DD,N_M[>0])\to\Hom_{\TT}(\DD,M[>1])=0
\end{eqnarray*}
and so $\Hom_{\TT}(\DD,N_M[>0])=0$.
Applying $\Hom_{\TT}(-,\DD[>0])$ to \eqref{exchange}, we have an exact sequence
\begin{eqnarray*}%\label{apply (-,M[i])}
\Hom_{\TT}(D[1],\DD[>0])\xrightarrow{\cdot f[1]}
\Hom_{\TT}(M[1],\DD[>0])\to\Hom_{\TT}(N_M,\DD[>0])\to\Hom_{\TT}(D,\DD[>0])=0.
\end{eqnarray*}
Since the above $(\cdot f[1])$ is surjective, we have $\Hom_{\TT}(N_M,\DD[>0])=0$. 
Applying $\Hom_{\TT}(-,M[>1])$ to \eqref{exchange}, we have an exact sequence
\begin{equation*}%\label{apply (-,X[>1])}
0=\Hom_{\TT}(M[1],M[>1])\to\Hom_{\TT}(N_M,M[>1])\to\Hom_{\TT}(D,M[>1])=0
\end{equation*}
and so $\Hom_{\TT}(N_M,M[>1])=0$.
Applying $\Hom_{\TT}(N_M,-[>0])$ to \eqref{exchange}, we have an exact sequence
\begin{equation*}%\label{apply (Y_X,-)}
0=\Hom_{\TT}(N_M,D[>0])\to\Hom_{\TT}(N_M,N_M[>0])\to\Hom_{\TT}(N_M,M[>1])=0
\end{equation*}
and so we have $\Hom_{\TT}(N_M,N_M[>0])=0$.

Consequently we have $\Hom_{\TT}(\mu^+(\MM;\DD),\mu^+(\MM;\DD)[>0])=0$.
\end{proof}

In general silting mutation of a tilting subcategory is not necessarily a tilting subcategory. 
For this, we have the following criterion. 

\begin{theorem}\label{when silting is tilting}
Let $\MM$ be a tilting subcategory of $\TT$.
\begin{itemize}
\item[(a)] For a covariantly finite subcategory $\DD$ of $\MM$, the following conditions are equivalent.
\begin{itemize}
\item[(i)] $\mu^+(\MM;\DD)$ is tilting.
\item[(ii)] Any $M\in\MM$ has a left $\DD$-approximation $f$ such that
$\homo{\TT}{\DD}{f}$ is injective.
\end{itemize}
\item[(b)] For a contravariantly finite subcategory $\DD$ of $\MM$, the following conditions are equivalent.
\begin{itemize}
\item[(i)] $\mu^-(\MM;\DD)$ is tilting.
\item[(ii)] Any $M\in\MM$ has a right $\DD$-approximation $g$ such that
$\homo{\TT}{g}{\DD}$ is injective.
%$\homo{\TT}{g}{\MM_{\XX}}$ is a monomorphism where $g$ is a minimal right $\MM_{\XX}$-approximation of $X$ for any $X\in \XX$. 
\end{itemize}
\end{itemize}
In these cases mutation is called \emph{tilting mutation}.
%We say that mutation is \emph{tilting mutation} if both $\MM$ and $\mu_{\XX}^{-}(\MM)$ ($\mu_{\XX}^{+}(\MM)$) are tilting.  
\end{theorem}

\begin{proof}
We only prove the statement (a). 
The proof is parallel to that of Theorem \ref{thm:mutation is silting}.

Applying $\Hom_{\TT}(\DD,-[\neq0])$ to \eqref{exchange}, we have an exact sequence
\begin{eqnarray*}%\label{apply (M,-)}
0=\Hom_{\TT}(\DD,D[\neq0])\to\Hom_{\TT}(\DD,N_M[\neq0])\to\Hom_{\TT}(\DD,M[\neq1])\xrightarrow{f\cdot}\Hom_{\TT}(\DD,D[\neq1]).
\end{eqnarray*}
Thus $\Hom_{\TT}(\DD,N_M[\neq0])=0$ holds if and only if (ii) holds.

Applying $\Hom_{\TT}(-,\DD[\neq0])$ to \eqref{exchange}, we have an exact sequence
\begin{eqnarray*}%\label{apply (-,M[i])}
\Hom_{\TT}(D[1],\DD[\neq0])\xrightarrow{\cdot f[1]}
\Hom_{\TT}(M[1],\DD[\neq0])\to\Hom_{\TT}(N_M,\DD[\neq0])\to\Hom_{\TT}(D,\DD[\neq0])=0.
\end{eqnarray*}
Since the above $(\cdot f[1])$ is surjective, we have $\Hom_{\TT}(N_M,\DD[\neq0])=0$.

%The implication $(i)\Rightarrow (ii)$ can be shown easily. 
%We shall show (i)$\Rightarrow$(ii). 
%Let $X\in \XX$ and put $Y=Y_{X}$ in the triangle \eqref{exchange}. 
%By the exact sequence \eqref{apply (-,M[i])}, we have $\homo{\TT}{Y}{\MM_{\XX}[i]}=0$ for any $i\not=0$. 
%Since the exact sequence \eqref{apply (M,-)} and $\homo{\TT}{\MM_{\XX}}{f}$ is a monomorphism, 
%we have $\homo{\TT}{\MM_{\XX}}{Y[i]}=0$ for any $i\not=0$. 

Applying $\Hom_{\TT}(M,-[\neq-1,0])$ to \eqref{exchange}, we have an exact sequence 
\[0=\homo{\TT}{M}{D[\neq-1,0]}\to \homo{\TT}{M}{N_M[\neq-1,0]}\to \homo{\TT}{M}{M[\neq0,1]}=0\]
and so $\homo{\TT}{M}{N_M[\neq-1,0]}=0$. 
Applying $\Hom_{\TT}(-,N_M[\neq0,1])$ to \eqref{exchange}, we have an exact sequence 
\[0=\homo{\TT}{M}{N_M[\neq-1,0]}\to \homo{\TT}{N_M}{N_M[\neq0,1]}\to \homo{\TT}{D}{N_M[\neq0,1]}=0\]
and so $\homo{\TT}{N_M}{N_M[\neq0,1]}=0$. 
Thus we have $\homo{\TT}{N_M}{N_M[\neq0]}=0$ by Theorem \ref{thm:mutation is silting}.

Consequently we have $\homo{\TT}{\mu^+(\MM;\DD)}{\mu^+(\MM;\DD)[\neq0]}=0$
if and only if (ii) holds. 
\end{proof}

The following properties can be checked easily.

\begin{proposition}\label{mutation basic}
Let $\MM\in\silt\TT$.
\begin{itemize}
\item[(a)] For a covariantly finite subcategory $\DD$ of $\MM$,
we have $\mu^-(\mu^+(\MM;\DD);\DD)=\MM$ and $\MM\ge\mu^+(\MM;\DD)$, where the equality holds if and only if $\DD=\MM$.
\item[(b)] For a contravariantly finite subcategory $\DD$ of $\MM$,
we have $\mu^+(\mu^-(\MM;\DD);\DD)=\MM$ and $\mu^-(\MM;\DD)\ge\MM$, where the equality holds if and only if $\DD=\MM$.
\end{itemize}
\end{proposition}

\begin{proof}
(a) Applying $\Hom_{\TT}(\DD,-)$ to \eqref{exchange}, we have an exact sequence
\[\Hom_{\TT}(\DD,D)\stackrel{g\cdot}{\to}\Hom_{\TT}(\DD,N_M)\to\Hom_{\TT}(\DD,M[1])=0.\]
Thus $g$ is a right $\DD$-approximation.
Thus we have $\mu^-(\mu^+(\MM;\DD);\DD)=\MM$.

Applying $\Hom_{\TT}(\MM,-[>0])$ to \eqref{exchange}, we have an exact sequence
\[0=\Hom_{\TT}(\MM,D[>0])\to\Hom_{\TT}(\MM,N_M[>0])\to\Hom_{\TT}(\MM,M[>1])=0\]
and so $\Hom_{\TT}(\MM,N_M[>0])=0$.
Thus $\Hom_{\TT}(\MM,\mu^+(\MM;\DD)[>0])=0$ holds and we have $\MM\ge\mu^+(\MM;\DD)$.

If $\DD\neq\MM$, then take $M\in\MM\backslash\DD$.
Since $f$ is not a split monomorphism, we have $\Hom_{\TT}(N_M,M[1])\neq0$.
Thus $N_M\notin\MM$ holds, so we have $\MM\neq\mu^+(\MM;\DD)$.
\end{proof}

In the rest of this subsection, we assume that $\TT$ is a Krull-Schmidt triangulated category.
The following notation will be often used.

\begin{definition}
Let $\MM\in\silt\TT$ and $\XX$ a subcategory of $\MM$.
Define a subcategory $\MM_{\XX}$ of $\MM$ by $\ind\MM_{\XX}=\ind\MM\backslash\ind\XX$.
%If $\MM_{\XX}$ is a covariantly (respectively, contravariantly) finite subcategory of $\MM$
We put
\[\mu^{\pm}_{\XX}(\MM):=\mu^{\pm}(\MM;\MM_{\XX})\]
%\ \ \ (\mbox{respectively, }\ \mu^-_{\XX}(\MM):=\nu^-(\MM;\MM_{\XX}))}\]
We say that mutation is \emph{irreducible} if $\#\ind\XX=1$.
If $\XX=\add X$, we denote $\MM_{\XX}$ and $\mu_{\XX}^\pm$ by $\MM_X$ and $\mu_X^\pm$ respectively.
\end{definition}

Now we assume the following condition:
\begin{itemize}
\item[(F)] $\TT$ is Krull-Schmidt, and for any $\MM\in\silt\TT$ and $X\in\ind\MM$, the subcategory $\MM_X$ is functorially finite in $\MM$.
\end{itemize}
This is satisfied if $\TT$ is $k$-linear $\Hom$-finite and has a silting object (Proposition \ref{existence of additive generator}).

Under this condition, we shall show the following result.

\begin{theorem}\label{two quivers}
Let $\TT$ be a triangulated category satisfying the condition (F).
For any $\MM,\NN\in\silt\TT$, the following conditions are equivalent.
\begin{itemize}
\item[(a)] $\NN$ is an irreducible left mutation of $\MM$.
\item[(b)] $\MM$ is an irreducible right mutation of $\NN$.
\item[(c)] $\MM>\NN$ and there is no $\LL\in\silt\TT$ satisfying $\MM>\LL>\NN$.
\end{itemize}
\end{theorem}

The following property is crucial in our consideration.

\begin{proposition}\label{make closer}
If $\MM,\NN\in\silt\TT$ satisfy $\MM>\NN$,
then there exists an irreducible left mutation $\LL$ of $\MM$ such that $\MM>\LL\ge\NN$.
\end{proposition}

\begin{proof}
We take $N_0\in\NN$ which does not belong to $\MM$,
then consider $\ell\ge0$ and triangles in Proposition \ref{resolution}.
Then we have $\ell>0$. Take an indecomposable summand $X$ of $M_\ell$,
and let $\LL:=\mu_X^+(\MM)$ be an irreducible left mutation.
Then we have $\LL=\add(\MM_X\cup\{Y\})$, where $Y$ is given by the triangle
\[X\xrightarrow{f}D\xrightarrow{g}Y\to X[1]\]
with a left $\MM_X$-approximation $f$ of $X$.

We only have to prove $\LL\ge\NN$. We need to show $\Hom_{\TT}(Y,\NN[>0])=0$.
Since we have an exact sequence
\[0=\Hom_{\TT}(X,\NN[>0])\to\Hom_{\TT}(Y,\NN[>1])\to\Hom_{\TT}(D,\NN[>1])=0,\]
we have $\Hom_{\TT}(Y,\NN[>1])=0$.
Thus it remains to show $\Hom_{\TT}(Y,\NN[1])=0$.
Since we have an exact sequence
\[\Hom_{\TT}(D,\NN)\xrightarrow{\cdot f}\Hom_{\TT}(X,\NN)\to\Hom_{\TT}(Y,\NN[1])\to
\Hom_{\TT}(D,\NN[1])=0,\]
we only have to show that $\Hom_{\TT}(D,\NN)\xrightarrow{\cdot f}\Hom_{\TT}(X,\NN)$ is surjective.

Fix any $N_0'\in\NN$ and $a:X\to N_0'$. 
We consider the diagram
\[\xymatrix@R=.4cm{
&Y[-1] \ar[r]& X \ar[r]^{f} \ar[d]^{a} & D\ar[r]^g&Y\\
N_1'\ar[r]^{g_0'}&M_0' \ar[r]^{f_0'} & N_0' \ar[r] &  N_1'[1].
}\]
where the lower triangle is given in Lemma \ref{no common summand}.
Since $f_0'$ is a right $\MM$-approximation, there exists $b:X\to M_0'$ such that $a=f_0'b$.
Since $(\add M_\ell)\cap(\add M_0')=0$ holds by Lemma \ref{no common summand},
we have $X\notin\add M_0'$ and so $M_0'\in\MM_X$.
Since $f$ is a left $\MM_X$-approximation, there exists $c:D\to M_0'$
such that $b=cf$. Thus we have $a=(f_0'c)f$, and we have the assertion.
\end{proof}

Now we are ready to prove Theorem \ref{two quivers}.

(i) We shall show that (a) and (b) are equivalent.

We only have to show that (a) implies (b).
Assume that $\NN$ is an irreducible left mutation $\mu_X^+(\MM)$ of $\MM$
with respect to $X\in\ind\MM$. Take a triangle
\[X\xrightarrow{f}D\xrightarrow{g}Y\to X[1]\]
with a minimal left $\MM_X$-approximation $f$. By Proposition \ref{mutation basic},
we have $\MM=\mu_Y^-(\NN)$. Since $g$ is a right $\MM_X$-approximation, it is easy to check that $Y$ is indecomposable.
Thus $\MM$ is an irreducible right mutation of $\NN$.

(ii) We shall show that (a) and (c) are equivalent.

Assume that (c) is satisfied.
Since $\MM>\NN$, there exists an irreducible left mutation $\LL$ of $\MM$ such that $\MM>\LL\ge\NN$
by Proposition \ref{make closer}. By the condition (c), we have $\LL=\NN$,
and we have the assertion.

Assume that (a) is satisfied. Thus $\NN=\mu_X^+(\MM)$ for indecomposable $X$.
Assume that there exists $\LL$ such that $\MM>\LL>\NN$.
Again by Proposition \ref{make closer}, there exists a left mutation
$\KK$ of $\MM$ such that $\MM>\KK\ge\LL$.
By Proposition \ref{common summand}, we have $\MM_X=\MM\cap\NN\subset\KK$.
Since both of $\NN$ and $\KK$ are left mutation of $\MM$ with respect to $X$,
we have $\NN=\KK$, a contradiction. Thus (c) holds.
\qed

%%%%%%%%%%%%%%%%%%%%%%%%%%%%%%%%%%%%%%%%%%%%%%%%%%%%%%%%%%%%%%%%%%%%%%%%%%%%%%%%%%%%%%%%%%%%%%%%%%%%%%%%%%%
\subsection{Silting reduction}
In this subsection, we give a reduction theorem of silting subcategories,
which is an analogue of 2-Calabi-Yau reduction in cluster tilting theory \cite[4.9]{IY}.
The following result gives a bijection between silting subcategories of $\TT$
containing a functorially finite thick subcategory $\SSS$ and silting subcategories of the quotient triangulated category $\TT/\SSS$.

\begin{theorem}\label{silting reduction}
Let $\TT$ be a Krull-Schmidt triangulated category, $\SSS$ a thick subcategory of $\TT$
and $\UU:=\TT/\SSS$. Let $F:\TT\to\UU$ be the canonical functor.
\begin{itemize}
\item[(a)] If $\SSS$ is a contravariantly finite subcategory of $\TT$, then for any
$\DD\in\silt\SSS$ we have an injective map
\[\{\MM\in\silt\TT\ |\ \DD\subset\MM\}\to\silt\UU,\ \ \ \MM\mapsto F\MM.\]
\item[(b)] If $\SSS$ is a functorially finite subcategory of $\TT$, then the map in (a) is bijective.
\end{itemize}
\end{theorem}

The first step of the proof is to consider the subcategory $\SSS^{\perp_{\TT}}$ of $\TT$.
Since $\SSS$ is contravariantly finite, we have a stable t-structure $(\SSS,\SSS^{\perp_{\TT}})$ of $\TT$
by Lemma \ref{direct summand}.
Moreover, we can naturally identify $\UU$ with the subcategory $\SSS^{\perp_{\TT}}$ of $\TT$ \cite{M}.

Next we shall show the following observation.

\begin{lemma}\label{S in S^le0}
Let $\MM$ be as in Theorem \ref{silting reduction}(a).
For any $M\in\MM$, take a triangle
\begin{equation}\label{split of M}
S\stackrel{a}{\to}M\stackrel{b}{\to}FM\to S[1]
\end{equation}
with $S\in\SSS$ and $FM\in\UU$. 
Then $S\in \SSS^{\leq 0}_{\DD}$ holds, where
\[({}^{\perp_{\SSS}}(\SSS^{\leq 0}_{\DD}),\SSS^{\leq0}_{\DD})
=(\bigcup_{\ell>0}\DD[-\ell]*\cdots*\DD[-1],\bigcup_{\ell\ge0}\DD*\cdots*\DD[\ell])\]
is the torsion pair in $\SSS=\thick\DD$ given in Proposition \ref{generate3}.
\end{lemma}

\begin{proof}
Take a triangle
\[S^{>0}\stackrel{c}{\to}S\stackrel{d}{\to}S^{\le0}\to S^{>0}[1]\]
with $S^{>0}\in {}^{\perp_{\SSS}}(\SSS^{\leq 0}_{\DD})$ and $S^{\le0}\in\SSS^{\leq 0}_{\DD}$.
Since $\Hom_{\TT}({}^{\perp_{\SSS}}(\SSS^{\leq 0}_{\DD}),\MM)=0$, we have $ac=0$.
Thus there exists a morphism $e:S^{\le0}\to M$ such that $a=ed$.
Since $be=0$ by $\Hom_{\DD}(\SSS,\UU)=0$, there exists $f:S^{\le0}\to S$ such that $e=af$.
\[\xymatrix@R=.4cm{
&S^{>0}\ar[d]^c\\
FM[-1]\ar[r]&S\ar[r]^a\ar@<.2em>[d]^d&M\ar[r]^b&FM\\
&S^{\le0}\ar@<.2em>[u]^f\ar[ru]_e
}\]
Since $a(1_S-fd)=0$, we have that $1_S-fd:S\to S$ factors through $FM[-1]$.
Since $\Hom_{\TT}(\SSS,\UU)=0$ again, we have $1_S=fd$.
Thus $d$ is a split monomorphism, and we have $S\in \SSS^{\leq 0}_{\DD}$.
\end{proof}

Now we are ready to prove Theorem \ref{silting reduction}.

(a)(i) We shall show that $F\MM$ is a silting subcategory of $\UU$.

We have $\thick F\MM\supset F\thick\MM=F\TT=\UU$.
Thus we only have to show $\Hom_{\TT}(F\MM,F\MM[>0])=0$.

Applying $\Hom_{\TT}(M,-[>0])$ to \eqref{split of M}, we have an exact sequence
\[0\stackrel{\MM\in\silt\TT}{=}\Hom_{\TT}(M,M[>0])\to\Hom_{\TT}(M,FM[>0])\to\Hom_{\TT}(M,S[>1])\stackrel{\Hom_{\TT}(\MM,\SSS^{<0}_{\DD})=0}{=}0\]
and so $\Hom_{\TT}(M,FM[>0])=0$.
Applying $\Hom_{\TT}(-,FM[>0])$ to \eqref{split of M}, we have an exact sequence
\[0\stackrel{\Hom_{\TT}(\SSS,\UU)=0}{=}\Hom_{\TT}(S,FM[>-1])\to\Hom_{\TT}(FM,FM[>0])\to\Hom_{\TT}(M,FM[>0])=0\]
and so $\Hom_{\TT}(FM,FM[>0])=0$.

(ii) We shall show that the correspondence $\MM\mapsto F\MM$ is injective. 
By Theorem \ref{partial order}, it is enough to prove $\MM\geq \NN$ if $F\MM\geq F\NN$. 
Let $M\in \MM$ and $N\in \NN$. 
For any $f\in \homo{\TT}{M}{N[>0]}$, we have a commutative diagram 
\[\xymatrix@R=0.4cm{
S_M \ar[r]\ar[d] & M \ar[r]\ar[d]^{f} & FM \ar[r]\ar[d] & S_M[1] \ar[d] \\
S_N[>0] \ar[r]_{a_N[>0]} & N[>0] \ar[r]_{b_N[>0]} & FN[>0] \ar[r] & S_N[>1].
}\]
Since $b_N[>0]\cdot f=0$ by our assumption, $f$ factors through $a_N[>0]: S_N[>0]\to N[>0]$. 
This implies $f=0$ since $\homo{\TT}{M}{S_N[>0]}\subset \homo{\TT}{\MM}{\SSS^{<0}_{\DD}}=0$ by Lemma \ref{S in S^le0}. 
Thus we have $\MM\geq \NN$. 

(b) We shall show that the correspondence $\MM\mapsto F\MM$ is surjective.
Fix $\NN\in\silt\UU$.

Since $\SSS$ is covariantly finite in $\TT$ by our assumption and
$\SSS^{<0}_{\DD}$ is covariantly finite in $\SSS$ by Proposition \ref{generate3}, 
we have that $\SSS^{<0}_{\DD}$ is covariantly finite in $\TT$.
Thus we have a torsion pair $({}^{\perp_{\TT}}(\SSS^{<0}_{\DD}), \SSS^{<0}_{\DD})$ in $\TT$ by Lemma \ref{direct summand}(c).
For any $N\in\NN$, take a triangle
\begin{equation}\label{split of N}
S_N\to M_N\to N\to S_N[1]
\end{equation}
with $M_N\in {}^{\perp_{\TT}}(\SSS^{<0}_{\DD})$ and $S_N\in \SSS^{\leq 0}_{\DD}$.
Notice that $M_N$ is unique up to a summand in $\DD$
since $\SSS^{\leq 0}_{\DD}\cap{}^{\perp_{\TT}}(\SSS^{<0}_{\DD})=\SSS^{\leq 0}_{\DD}\cap{}^{\perp_{\SSS}}(\SSS^{<0}_{\DD})=\DD$ by Proposition \ref{recover}.

We shall show that $\MM:=\add(\DD\cup\{M_N\ |\ N\in\NN\})$ is a silting subcategory of $\TT$.
Since $\DD[>0]\subset\SSS^{<0}_{\DD}$, we have $\Hom_{\TT}(M_N,\DD[>0])=0$.
Applying $\Hom_{\TT}(\DD,-[>0])$ to \eqref{split of N}, we have an exact sequence
\[0\stackrel{\Hom_{\TT}(\DD,\SSS^{<0}_{\DD})=0}{=}\Hom_{\TT}(\DD,S_N[>0])\to\Hom_{\TT}(\DD,M_N[>0])\to\Hom_{\TT}(\DD,N[>0])
\stackrel{\Hom_{\TT}(\SSS,\UU)=0}{=}0\]
and so $\Hom_{\TT}(\DD,M_N[>0])=0$.

Applying $\Hom_{\TT}(-,N[>0])$ to \eqref{split of N}, we have an exact sequence
\[0\stackrel{\NN\in\silt\UU}{=}\Hom_{\TT}(N,N[>0])\to\Hom_{\TT}(M_N,N[>0])\to\Hom_{\TT}(S_N,N[>0])\stackrel{\Hom_{\TT}(\SSS,\UU)=0}{=}0\]
and so $\Hom_{\TT}(M_N,N[>0])=0$.
Applying $\Hom_{\TT}(M_N,-[>0])$ to \eqref{split of N}, we have an exact sequence
\[0\stackrel{\Hom_{\TT}({}^{\perp_{\TT}}(\SSS^{<0}_{\DD}),\SSS^{<0}_{\DD})=0}{=}
\Hom_{\TT}(M_N,S_N[>0])\to\Hom_{\TT}(M_N,M_N[>0])\to\Hom_{\TT}(M_N,N[>0])=0\]
and so $\Hom_{\TT}(M_N,M_N[>0])=0$.

Consequently we have $\Hom_{\TT}(\MM,\MM[>0])=0$.
By \eqref{split of N}, we have $\thick\MM\supset\NN$, so we have $\thick\MM\supset\thick(\DD\cup\NN)\supset\SSS*\UU=\TT$.
Thus $\MM$ is a silting subcategory of $\TT$.
%
%(ii) It remains to show that the correspondence $\MM\mapsto F\MM$ is injective.
%
%We only have to show that, for any $M\in\MM$, the object $M_{FM}$ given
%by the triangle \eqref{split of N} for $N:=FM$ is isomorphic to $M$ up to a summand in $\DD$.
%This is clear since the triangle \eqref{split of M} gives the triangle
%\eqref{split of N} since $M\in {}^{\perp_{\TT}}(\SSS^{<0}_{\DD})$ and
%$S\in\SSS^{\leq 0}_{\DD}$ hold by Lemma \ref{S in S^le0}.
\qed

\medskip
The functor $F:\MM\to F\MM$ has the following properties.

\begin{proposition}\label{M and FM}
In Theorem \ref{silting reduction}(a), the functor $F:\MM\to F\MM$ induces an equivalence
\[\MM/[\DD]\to F\MM,\]
where $[\DD]$ is the ideal of $\MM$ consisting of morphisms which factor through objects in $\DD$.
In particular, $F$ induces a bijection between $\ind\MM\backslash\ind\DD$ and $\ind F\MM$.
\end{proposition}

\begin{proof}
Since $F\DD=0$, we have the induced functor $F:\MM/[\DD]\to F\MM$.
We only have to show that this is fully faithful.
Let $M,M'\in\MM$. For each $f:M\to M'$, the morphism $Ff:FM\to FM'$ is given by the following commutative diagram.
\[\xymatrix@R=.4cm{
S\ar[r]\ar[d]&M\ar[r]\ar[d]^f&FM\ar[r]\ar[d]^{Ff}&S[1]\ar[d]\\
S'\ar[r]&M'\ar[r]&FM'\ar[r]&S'[1]
}\]

(i) We shall show that $F:\MM/[\DD]\to F\MM$ is faithful.

If $Ff=0$, then $f$ factors through $S'$. By Lemma \ref{S in S^le0}, we know
$S'\in\SSS^{\leq0}_{\DD}=\bigcup_{\ell\ge0}\DD*\cdots*\DD[\ell]$.
Since $\Hom_{\TT}(M,\DD[1]*\cdots*\DD[\ell])=0$, we have that $f$ factors through $\DD$.
Thus $f$ is zero in $\MM/[\DD]$.

(ii) We shall show that $F:\MM/[\DD]\to F\MM$ is full.

Fix $g\in\Hom_{\TT}(FM,FM')$. Since $\Hom_{\TT}(M,S'[1])=0$ again by Lemma \ref{S in S^le0}, there exists $f:M\to M'$ which makes the following diagram commutative.
\[\xymatrix@R=.4cm{
S\ar[r]\ar[d]&M\ar[r]\ar[d]^f&FM\ar[r]\ar[d]^{g}&S[1]\ar[d]\\
S'\ar[r]&M'\ar[r]&FM'\ar[r]&S'[1]
}\]
Then $f=Fg$ holds.
\end{proof}

Later we use the following compatibility of silting mutation and silting reduction.

\begin{lemma}\label{compatibility}
In Theorem \ref{silting reduction}(a), we have $F\mu^+_{\XX}(\MM)=\mu^+_{F\XX}(F\MM)$ for any covariantly finite subcategory $\XX$ of $\MM$ satisfying $\DD\cap\XX=0$.
\end{lemma}

\begin{proof}
By Proposition \ref{M and FM}, we have $(F\MM)_{F\XX}=F(\MM_{\XX})$.
For any $X\in\XX$, take a triangle
\[X\stackrel{f}{\to} D\to Y\to X[1]\]
where $f$ is a left $\MM_{\XX}$-approximation. 
We have to prove that $Ff$ is a left $F(\MM_{\XX})$-approximation.
It is enough to show $\homo{\TT}{FY}{F(\MM_{\XX})[1]}=0$.  

For any $M\in \MM_{\XX}$, there exists a triangle 
\begin{equation}%\label{prop:compat 1}
S \to M \to FM \to S[1]
\end{equation}
with $S\in \SSS^{\leq 0}_{\DD}$ by Lemma \ref{S in S^le0}.
Applying $\homo{\TT}{Y}{-}$, we have an exact sequence 
\[
0 \stackrel{{\rm Thm.}\ \ref{thm:mutation is silting}}{=}  \homo{\TT}{Y}{M[1]}\to \homo{\TT}{Y}{FM[1]} \to 
\homo{\TT}{Y}{S[2]}\stackrel{S\in \SSS^{\leq 0}_{\DD}}{=}0,
\]
so we have $\homo{\TT}{Y}{FM[1]}=0$.
Thus we have $\homo{\TT}{FY}{FM[1]}\simeq \homo{\TT}{Y}{FM[1]}=0$. 
\end{proof}

%%%%%%%%%%%%%%%%%%%%%%%%%%%%%%%%%%%%%%%%%%%%%%%%%%%%%%%%%%%%%%%%%%%%%%%%%%%%%%%%%%%%%%%%%%%%%%%%%%%%%%%%%%
%%%%%%%%%%%%%%%%%%%%%%%%%%%%%%%%%%%%%%%%%%%%%%%%%%%%%%%%%%%%%%%%%%%%%%%%%%%%%%%%%%%%%%%%%%%%%%%%%%%%%%%%%%%

\subsection{Silting quivers and examples}

Let $\TT$ be a Krull-Schmidt triangulated category.
The aim of this section is to introduce the silting quiver of $\TT$.

\begin{definition}
The \emph{silting quiver} of $\TT$ is defined as follows:
\begin{itemize}
\item The set of vertices is $\silt\TT$.
\item We draw an arrow $\MM\to\NN$ if $\NN$ is an irreducible left mutation of $\MM$.
\end{itemize}
If $\TT$ satisfies the condition (F), then the silting quiver is nothing but the Hasse quiver of the partially ordered set $\silt\TT$ by Theorem \ref{two quivers}.
\end{definition}

We pose the following question.

\begin{question}\label{transitive question2}
Let $A$ be a finite dimensional algebra $A$ over a field and $\TT:=\K^{\rm b}(\proj A)$.
When is the silting quiver of $\TT$ connected?
Equivalently, when is the action of iterated irreducible mutation on $\silt\TT$ transitive?
\end{question}

If $A$ is local, then we have a positive answer by the following observation.

\begin{corollary}\label{corollary indecomposable case}
If $\TT$ has an indecomposable silting object, 
then the action of iterated irreducible mutation on $\silt\TT$ is transitive.
\end{corollary}

\begin{proof}
This is immediate from Theorem \ref{indecomposable case}
since we have $\mu_{M[i]}^{\pm}(M[i])=M[i\pm1]$.
\end{proof}

Next we shall generalize Corollary \ref{corollary indecomposable case} by
introducing the following notion which generalizes `almost complete partial tilting modules'.

We call a subcategory $\DD$ of $\TT$ \emph{almost complete silting} if there exists a
silting subcategory $\MM$ of $\TT$ such that $\MM\supset\DD$ and $\#(\ind\MM\backslash\ind\DD)=1$.

Then we have the following analogue of \cite[2.1]{HU} for `tilting mutation'
and \cite[5.3]{IY} for `cluster tilting mutation'.

\begin{theorem}\label{almost transitivity}
Let $\TT$ be a Krull-Schmidt triangulated category and $\DD$ an almost complete silting subcategory.
If $\thick\DD$ is a functorially finite subcategory of $\TT$, then the set $\{\MM\in\silt\TT\ |\ \DD\subset\MM\}$ is transitive under iterated irreducible mutation.
\end{theorem}

\begin{proof}
Let $\UU:=\TT/\thick\DD$.
By Theorem \ref{silting reduction} we have a bijection
$\{\MM\in\silt\TT\ |\ \DD\subset\MM\}\to\silt\UU$.
This commutes with mutation by Lemma \ref{compatibility} and its dual.
By Corollary \ref{corollary indecomposable case}, we know that $\silt\UU$ is transitive under iterated irreducible mutation.
Thus the assertion follows.
\end{proof}

We give examples of silting quivers.

\begin{example}
Let $A$ be a path algebra of the quiver $\xymatrix{1\ar[r]&2}$. 
We have the AR-quiver of $\K^{\rm b}(\proj{A})$ as follows:
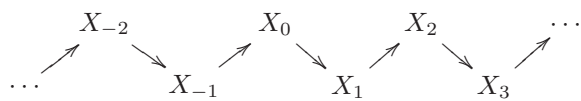
\[\xymatrix@R.3cm@C.3cm{
& X_{-2} \ar[dr] &                & X_{0} \ar[dr] &               & X_{2} \ar[dr] & &\cdots\\
\cdots\ar[ur]&                & X_{-1} \ar[ur] &               & X_{1} \ar[ur] &               & X_3\ar[ur]
}\]

Then the silting quiver of $A$ is the following (cf. Theorem \ref{hereditary case}):
\[
\xymatrix@R=0.5cm{
\cdots\ar[dd]\ar[rddd]&&\cdots\ar[lddd]\ar[rddd]&&\cdots\ar[lddd]\\
&X_{-3}\oplus X_1\ar[lddd]\ar[rddd]&&X_{-6}\oplus X_4\ar[lddd]\ar[rddd]&\\
X_{-1}\oplus X_0\ar[dd]\ar[rddd]&&X_{-4}\oplus X_3\ar[lddd]\ar[rddd]&&\cdots\ar[lddd]\\
&X_{-2}\oplus X_2\ar[lddd]\ar[rddd]&&X_{-5}\oplus X_5\ar[lddd]\ar[rddd]&\\
X_0\oplus X_1\ar[dd]\ar[rddd]&&X_{-3}\oplus X_4\ar[lddd]\ar[rddd]&&\cdots\ar[lddd]\\
&X_{-1}\oplus X_3\ar[lddd]\ar[rddd]&&X_{-4}\oplus X_6\ar[lddd]\ar[rddd]&\\
X_1\oplus X_2\ar[dd]\ar[rddd]&&X_{-2}\oplus X_5\ar[lddd]\ar[rddd]&&\cdots\ar[lddd]\\
&X_0\oplus X_4\ar[lddd]\ar[rddd]&&X_{-1}\oplus X_7\ar[lddd]\ar[rddd]&\\
X_2\oplus X_3\ar[dd]&&X_{-1}\oplus X_6&&\\
&X_1\oplus X_5&&X_{-2}\oplus X_8&\\
\cdots &&\cdots&&\cdots
}
\]
Identifying each silting object $T$ with $T[i]$ for any $i\in\Z$, we can simplify the quiver as follows,
where $\xymatrix{M\ar[r]|{[1]}&N}$ means $\xymatrix{M\ar[r]&N[1]}$.
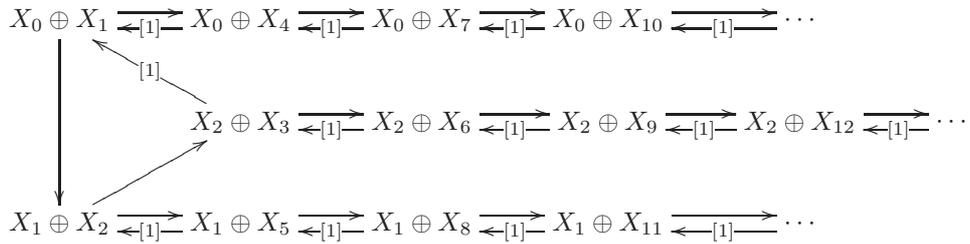
\[
\xymatrix{
X_0\oplus X_1\ar[dd]\ar@<.3em>[r]&X_0\oplus X_4\ar@<.3em>[r]\ar@<.3em>[l]|{[1]}&X_0\oplus X_7\ar@<.3em>[r]\ar@<.3em>[l]|{[1]}
&X_0\oplus X_{10}\ar@<.3em>[r]\ar@<.3em>[l]|{[1]}&\cdots\ar@<.3em>[l]|(0.4){[1]}\\
&X_2\oplus X_3\ar[lu]|{[1]}\ar@<.3em>[r]&X_2\oplus X_6\ar@<.3em>[r]\ar@<.3em>[l]|{[1]}
&X_2\oplus X_9\ar@<.3em>[r]\ar@<.3em>[l]|{[1]}&X_2\oplus X_{12}\ar@<.3em>[r]\ar@<.3em>[l]|{[1]}&\cdots\ar@<.3em>[l]|(0.35){[1]}\\
X_1\oplus X_2\ar[ru]\ar@<.3em>[r]&X_1\oplus X_5\ar@<.3em>[r]\ar@<.3em>[l]|{[1]}&X_1\oplus X_8\ar@<.3em>[r]\ar@<.3em>[l]|{[1]}
&X_1\oplus X_{11}\ar@<.3em>[r]\ar@<.3em>[l]|{[1]}&\cdots\ar@<.3em>[l]|(0.4){[1]}
}
\]
\end{example}

\begin{example}
Let $A$ be a path algebra of the quiver 
$\xymatrix{
1\ar@<.5em>[r]^{}="a" & 2 \ar@{<-}@<.5em>[l]^{}="b"  
\ar@{.}"a";"b"
}$ 
with $\ell\ge2$ arrows. 
The AR-quiver of $\K^{\rm b}(\proj{A})$ contains the following connected component:
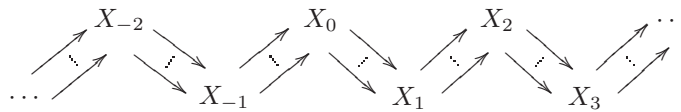
\[\xymatrix@R.5cm@C.5cm{
& X_{-2}\ar@<-0.7em>[dr]^{}="a1"\ar@<0.5em>[dr]_{}="a2" & & X_{0} \ar@<-0.7em>[dr]^{}="b1"\ar@<0.5em>[dr]_{}="b2" 
& & X_{2} \ar@<-0.7em>[dr]^{}="c1"\ar@<0.5em>[dr]_{}="c2" & &\cdots\\
\cdots\ar@<0.5em>[ur]_{}="g1"\ar@<-0.7em>[ur]^{}="g2" & & X_{-1} \ar@<0.5em>[ur]_{}="d1"\ar@<-0.7em>[ur]^{}="d2" & & X_{1} \ar@<0.5em>[ur]_{}="e1"\ar@<-0.7em>[ur]^{}="e2" & & 
X_3\ar@<0.5em>[ur]_{}="f1"\ar@<-0.7em>[ur]^{}="f2"
\ar@{.}"a1";"a2" \ar@{.}"b1";"b2" \ar@{.}"c1";"c2" \ar@{.}"d1";"d2" \ar@{.}"e1";"e2" \ar@{.}"f1";"f2" \ar@{.}"g1";"g2"
}\]
Then the silting quiver of $A$ is the following (cf. Theorem \ref{hereditary case}):
\[
\xymatrix{
\cdots\ar[d]\\
X_{-1}\oplus X_0\ar[d]\ar@<.3em>[r]&X_{-1}\oplus X_0[1]\ar@<.3em>[r]\ar@<.3em>[l]|{[1]}
&X_{-1}\oplus X_0[2]\ar@<.3em>[r]\ar@<.3em>[l]|{[1]}&X_{-1}\oplus X_0[3]\ar@<.3em>[r]\ar@<.3em>[l]|{[1]}
&\cdots\ar@<.3em>[l]|(0.3){[1]}\\
X_0\oplus X_1\ar[d]\ar@<.3em>[r]&X_0\oplus X_1[1]\ar@<.3em>[r]\ar@<.3em>[l]|{[1]}&X_0\oplus X_1[2]\ar@<.3em>[r]\ar@<.3em>[l]|{[1]}
&X_0\oplus X_1[3]\ar@<.3em>[r]\ar@<.3em>[l]|{[1]}&\cdots\ar@<.3em>[l]|(0.35){[1]}\\
X_1\oplus X_2\ar[d]\ar@<.3em>[r]&X_1\oplus X_2[1]\ar@<.3em>[r]\ar@<.3em>[l]|{[1]}&X_1\oplus X_2[2]\ar@<.3em>[r]\ar@<.3em>[l]|{[1]}
&X_1\oplus X_2[3]\ar@<.3em>[r]\ar@<.3em>[l]|{[1]}&\cdots\ar@<.3em>[l]|(0.35){[1]}\\
X_2\oplus X_3\ar[d]\ar@<.3em>[r]&X_2\oplus X_3[1]\ar@<.3em>[r]\ar@<.3em>[l]|{[1]}&X_2\oplus X_3[2]\ar@<.3em>[r]\ar@<.3em>[l]|{[1]}
&X_2\oplus X_3[3]\ar@<.3em>[r]\ar@<.3em>[l]|{[1]}&\cdots\ar@<.3em>[l]|(0.35){[1]}\\
\cdots\\
}\]
\end{example}

\begin{example}
Let $A$ be an algebra presented by a quiver $\xymatrix{1\ar^a@<.3em>[r]&\ar^b@<.3em>[l]2}$
with relations $ab=0=ba$.
Then the silting quiver of $A$ is the following, where $X:={\rm cone}(P_1\to P_2)$ and $Y:={\rm cone}(P_2\to P_1)$:
\[
\xymatrix{
&X\oplus P_2\ar@<.3em>[r]\ar[d]&X[1]\oplus P_2\ar@<.3em>[r]\ar@<.3em>[l]|{[1]}
&X[2]\oplus P_2\ar@<.3em>[r]\ar@<.3em>[l]|{[1]}&X[3]\oplus P_2\ar@<.3em>[r]\ar@<.3em>[l]|{[1]}&\cdots\ar@<.3em>[l]|(0.35){[1]}\\
&X\oplus P_1[1]\ar@<.3em>[r]\ar[ld]|{[1]}&X\oplus P_1[2]\ar@<.3em>[r]\ar@<.3em>[l]|{[1]}
&X\oplus P_1[3]\ar@<.3em>[r]\ar@<.3em>[l]|{[1]}&X\oplus P_1[4]\ar@<.3em>[r]\ar@<.3em>[l]|{[1]}&\cdots\ar@<.3em>[l]|(0.35){[1]}\\
P_1\oplus P_2\ar[ruu]\ar[rd]\\
&Y\oplus P_1\ar@<.3em>[r]\ar[d]&Y[1]\oplus P_1\ar@<.3em>[r]\ar@<.3em>[l]|{[1]}
&Y[2]\oplus P_1\ar@<.3em>[r]\ar@<.3em>[l]|{[1]}&Y[3]\oplus P_1\ar@<.3em>[r]\ar@<.3em>[l]|{[1]}&\cdots\ar@<.3em>[l]|(0.35){[1]}\\
&Y\oplus P_2[1]\ar@<.3em>[r]\ar[luu]|{[1]}&Y\oplus P_2[2]\ar@<.3em>[r]\ar@<.3em>[l]|{[1]}
&Y\oplus P_2[3]\ar@<.3em>[r]\ar@<.3em>[l]|{[1]}&Y\oplus P_2[4]\ar@<.3em>[r]\ar@<.3em>[l]|{[1]}&\cdots\ar@<.3em>[l]|(0.35){[1]}
}\]
\end{example}

\begin{example}
Let $A$ be an algebra presented by a quiver $\xymatrix{1\ar^a@<.3em>[r]&\ar^b@<.3em>[l]2}$
with relations $(ab)^{\ell}a=0=(ba)^{\ell}b$ ($\ell\ge1$).
The silting quiver of $A$ is connected by \cite{A2}, and it is the following:
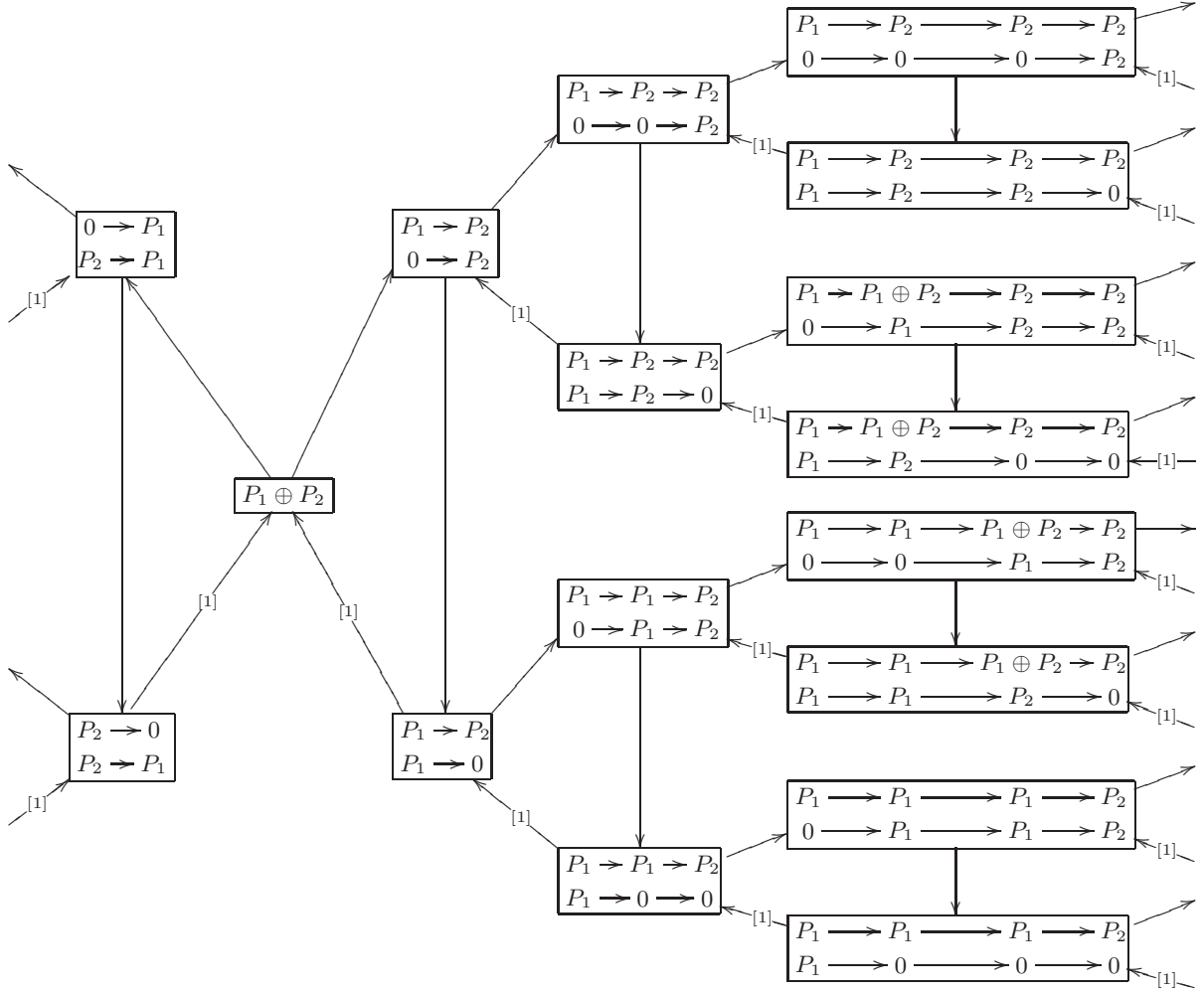
\[
{\small{
\xymatrix@R=0.1pt @C=0.3cm{
&& & && && & && & & && & & & &&\\
&& & && && & && & & && P_1 \ar[r] & P_2 \ar[r] & P_2 \ar[r] & P_2 &&\\
&& & && && & && & & && 0 \ar[r] & 0 \ar[r] & 0 \ar[r] & P_2 &&\\
&& & && && & && P_1 \ar[r] & P_2 \ar[r] & P_2 && & & & &&\\
&& & && && & && 0 \ar[r] & 0 \ar[r] & P_2 && & & & &&\\   
&& & && && & && & & && P_1 \ar[r] & P_2 \ar[r] & P_2 \ar[r] & P_2 &&\\
&& & && && & && & & && P_1 \ar[r] & P_2 \ar[r] & P_2 \ar[r] & 0 &&\\
&& 0 \ar[r] & P_1 && && P_1 \ar[r] & P_2 && & & && & & & &&\\
&& P_2 \ar[r] & P_1 && && 0 \ar[r] & P_2 && & & && & & & &&\\
&& & && && & && & & && P_1 \ar[r] & P_1\oplus P_2 \ar[r] & P_2 \ar[r] & P_2 &&\\
&& & && && & && & & && 0 \ar[r] & P_1 \ar[r] & P_2 \ar[r] & P_2 &&\\
&& & && && & && P_1 \ar[r] & P_2 \ar[r] & P_2 && & & & &&\\
&& & && && & && P_1 \ar[r] & P_2 \ar[r] & 0 && & & & &&\\
&& & && && & && & & && P_1 \ar[r] & P_1\oplus P_2 \ar[r] & P_2 \ar[r] & P_2 &&\\
&& & && && & && & & && P_1 \ar[r] & P_2 \ar[r] & 0 \ar[r] & 0 &&\\
&& & && P_1\oplus P_2 && & && & & && & & & &&\\
&& & && && & && & & && P_1 \ar[r] & P_1 \ar[r] & P_1\oplus P_2 \ar[r] & P_2 &&\\
&& & && && & && & & && 0 \ar[r] & 0 \ar[r] & P_1 \ar[r] & P_2 &&\\
&& & && && & && P_1 \ar[r] & P_1 \ar[r] & P_2 && & & & &&\\
&& & && && & && 0 \ar[r] & P_1 \ar[r] & P_2 && & & & &&\\
&& & && && & && & & && P_1 \ar[r] & P_1 \ar[r] & P_1\oplus P_2 \ar[r] & P_2 &&\\
&& & && && & && & & && P_1 \ar[r] & P_1 \ar[r] & P_2 \ar[r] & 0 &&\\
&& P_2 \ar[r] & 0   && && P_1 \ar[r] & P_2 && & & && & & & &&\\
&& P_2 \ar[r] & P_1 && && P_1 \ar[r] & 0 && & & && & & & &&\\
&& & && && & && & & && P_1 \ar[r] & P_1 \ar[r] & P_1 \ar[r] & P_2 &&\\
&& & && && & && & & && 0 \ar[r] & P_1 \ar[r] & P_1 \ar[r] & P_2 &&\\
&& & && && & && P_1 \ar[r] & P_1 \ar[r] & P_2 && & & & &&\\
&& & && && & && P_1 \ar[r] & 0 \ar[r] & 0 && & & & &&\\
&& & && && & && & & && P_1 \ar[r] & P_1 \ar[r] & P_1 \ar[r] & P_2 &&\\
&& & && && & && & & && P_1 \ar[r] & 0 \ar[r] & 0 \ar[r] & 0 &&\\
&& & && && & && & & && & & & &&
\save "2,15"."3,18"*[F]\frm{}="a" \restore
\save "6,15"."7,18"*[F]\frm{}="b" \restore
\save "10,15"."11,18"*[F]\frm{}="c" \restore
\save "14,15"."15,18"*[F]\frm{}="d" \restore
\save "17,15"."18,18"*[F]\frm{}="e" \restore
\save "21,15"."22,18"*[F]\frm{}="f" \restore
\save "25,15"."26,18"*[F]\frm{}="g" \restore
\save "29,15"."30,18"*[F]\frm{}="h" \restore
\save "4,11"."5,13"*[F]\frm{}="i" \restore
\save "12,11"."13,13"*[F]\frm{}="j" \restore
\save "19,11"."20,13"*[F]\frm{}="k" \restore
\save "27,11"."28,13"*[F]\frm{}="l" \restore
\save "8,8"."9,9"*[F]\frm{}="m" \restore
\save "23,8"."24,9"*[F]\frm{}="n" \restore
\save "16,6"*[F]\frm{}="o" \restore
\save "8,3"."9,4"*[F]\frm{}="p" \restore
\save "23,3"."24,4"*[F]\frm{}="q" \restore
\ar@<6em> "a";"b"
\ar@<6em> "c";"d"
\ar@<6em> "e";"f"
\ar@<6em> "g";"h"
\ar@<2.5em> "i";"j"
\ar@<2.5em> "k";"l"
\ar@<1.3em> "m";"n"
\ar@<1.3em> "p";"q"
\ar "2,18";"1,20" 
\ar "6,18";"5,20" 
\ar "10,18";"9,20" 
\ar "14,18";"13,20" 
\ar "17,18";"17,20" 
\ar "21,18";"20,20" 
\ar "25,18";"24,20" 
\ar "29,18";"28,20" 
\ar "4,20";"3,18" |(0.4){[1]} 
\ar "8,20";"7,18" |(0.4){[1]} 
\ar "12,20";"11,18" |(0.4){[1]} 
\ar "15,20";"15,18" |(0.4){[1]} 
\ar "19,20";"18,18" |(0.4){[1]} 
\ar "23,20";"22,18" |(0.4){[1]} 
\ar "27,20";"26,18" |(0.4){[1]} 
\ar "31,20";"30,18" |(0.4){[1]} 
\ar "4,13";{"3,15"+<-0.9em, -0.1em>}
\ar "12,13";{"11,15"+<-0.9em, -0.1em>}
\ar "19,13";{"18,15"+<-0.9em, -0.1em>}
\ar "27,13";{"26,15"+<-0.9em, -0.1em>}
\ar "b";"i" |(0.2){[1]}
\ar "d";"j" |(0.2){[1]}
\ar "f";"k" |(0.2){[1]}
\ar "h";"l" |(0.2){[1]}
\ar "8,9";{"5,11"+<-0.9em, -0.4em>}
\ar "23,9";{"20,11"+<-0.9em, -0.4em>}
\ar "j";"m" |(0.35){[1]}
\ar "l";"n" |(0.35){[1]}
\ar "o";"m"
\ar "n";"o" |{[1]}
\ar "o";"p"
\ar {"q"+<1.7em,0.9em>};"o" |{[1]}
\ar "8,3";"6,1"
\ar "23,3";"21,1"
\ar "11,1";"9,3" |(0.4){[1]}
\ar "26,1";"24,3" |(0.4){[1]}
}}}
\]
\end{example}

%%%%%%%%%%%%%%%%%%%%%%%%%%%%%%%%%%%%%%%%%%%%%%%%%%%%%%%%%%%%%%%%%%%%%%%%%%%%%%%%%%%%%%%%%%%%%%%%%%%%%%%%%%%
%%%%%%%%%%%%%%%%%%%%%%%%%%%%%%%%%%%%%%%%%%%%%%%%%%%%%%%%%%%%%%%%%%%%%%%%%%%%%%%%%%%%%%%%%%%%%%%%%%%%%%%%%%%

\subsection{Okuyama-Rickard complexes and APR and BB tilting modules} \label{Examples of silting subcategories and objects}

Our silting mutation has two important origins in representation theory: 
\begin{itemize}
\item Okuyama-Rickard complexes and Okuyama's method in modular representation theory;
\item APR tilting modules and BB tilting modules;
%\item Riedtmann-Schofield and Happel-Unger theory on tilting mutation.
\end{itemize}
In this section we will explain the relationship between silting mutation and these notions.

Throughout this section, let $A$ be a finite dimensional algebra over a field $k$. 
We may assume that $A$ is basic and indecomposable as an $A$-$A$-bimodule.
For an $A$-module $X$, we denote by $P(X)$ a projective cover of $X$.
We denote by $\nu =D\homo{A}{-}{A}:\modu{A}\to\modu{A}$ the Nakayama functor.
We denote by $\tau$ and $\tau^{-1}$ the Auslander-Reiten translations \cite{ARS,ASS}. 

%\subsection{Okuyama-Rickard complexes and Okuyama's method}

\begin{definition}
For idempotent $e\in A$,
%Set $U$ as the $(-1)$-shift of a minimal projective presentation of $eA/eA(1-e)A$. 
the \emph{Okuyama-Rickard complex} with respect to $e$ is defined by
\begin{equation}\label{OR complex}
T:=\left\{\begin{array}{ccc}
\stackrel{0}{P(eA(1-e)A)}&\stackrel{p_e}{\longrightarrow}&\stackrel{1}{eA}\\
\oplus&&\\
(1-e)A&&
\end{array}\right.
\end{equation}
where $p_e$ gives a projective cover of the submodule $eA(1-e)A$ of $eA$.
%Now we define $T=U\oplus (1-e)A$ and call it the \emph{Okuyama-Rickard complex} with respect to $e$. 
\end{definition}

In \cite{O}, Okuyama constructed Okuyama-Rickard complexes and proved that it is tilting if $A$ is symmetric. 
The method of his construction is often called \emph{Okuyama's method}. 

Let us give basic properties of Okuyama-Rickard complexes in our context of silting mutation.

\begin{theorem}\label{prop:1_right_mu}
Let $e\in A$ be an idempotent and $T$ the Okuyama-Rickard complex with respect to $e$. 
\begin{itemize}
\item[(a)] $T$ is isomorphic to right mutation $\mu^{-}_{eA}(A)$ of $A$ with respect to $eA$.
\item[(b)] $T$ is a silting object in $\K^{\rm{b}}(\proj{A})$.
\item[(c)] The following conditions are equivalent.
\begin{itemize}
\item[(i)] $T$ is a tilting object in $\K^{\rm{b}}(\proj{A})$.
\item[(ii)] $\homo{A}{eA/eA(1-e)A}{(1-e)A}=0$.
\end{itemize}
\item[(d)] If $A$ is a self-injective algebra and $eA\simeq\nu (eA)$ holds, then $T$ is a tilting object in $\K^{\rm{b}}(\proj{A})$.
\end{itemize}
\end{theorem}

\begin{proof}
(a) Since $eA(1-e)A$ is minimal amongst submodules $X$ of $eA$ such that any composition factor of $eA/X$ belongs to $\add({\rm top}\ eA)$,
we have that $p_e$ in \eqref{OR complex} is a minimal right $\add{(1-e)A}$-approximation of $eA$. 
Thus the assertion follows.

(b)(c) These are immediate consequences of Theorems \ref{thm:mutation is silting} and \ref{when silting is tilting}.

(d) This is immediate from (c) since $\add({\rm top}\ eA)\cap\add({\rm soc}(1-e)A)=\add({\rm top}\ eA)\cap\add({\rm top} (1-e)A)=0$ holds by our assumption
and any composition factor of $eA/eA(1-e)A$ is isomorphic to ${\rm top}\ eA$.
\end{proof}

The partial order $\le$ on $\silt{\K^{\rm{b}}(\proj{A})}$
have the following easy interpretation for Okuyama-Rickard complexes. 

\begin{proposition}\label{prop:relation of OR}
Let $e, e'\in A$ be idempotents and $T, T'$ the Okuyama-Rickard complexes with respect to $e, e'$ respectively. 
Then the following conditions are equivalent.
\begin{itemize}
\item[(a)] $\add{eA} \supseteq \add{e'A}$.
\item[(b)] $T \ge T'$.   
\end{itemize}
\end{proposition}

\begin{proof}
Assume $\add{eA} \supseteq \add{e'A}$. 
Clearly we have $\Hom_{\K^{\rm{b}}(\proj{A})}(T,T'[>1])=0$. 
Take any morphism
\[\xymatrix{
T\ar[d]&0\ar[r]\ar[d]&P(eA(1-e)A)\oplus(1-e)A\ar^(0.8){p_{e}}[r]\ar^{f}[d]&eA\ar[d]\\
T'[1]&P(e'A(1-e')A)\oplus(1-e')A\ar^(0.7){p_{e'}}[r]&e'A\ar[r]&0
}\]
Since $P(eA(1-e)A)\oplus(1-e)A\in\add{(1-e)A}\subset\add{(1-e')A}$
and $p_{e'}$ is a right $\add{(1-e')A}$-approximation, $f$ factors through $p_{e'}$.
Thus we have $\Hom_{\K^{\rm{b}}(\proj{A})}(T,T'[1])=0$, so $T\ge T'$. 

Assume $T\ge T'$. Then we have $\Hom_{\K^{\rm{b}}(\proj{A})}(T,T'[1])=0$.
In particular, for any morphism $f:P(eA(1-e)A)\oplus(1-e)A\to e'A$,
there exist $s$ and $t$ such that $f=p_{e'}s+tp_e$.
Since both $p_e$ and $p_{e'}$ belongs to $J_{\modu{A}}$, any morphism
in $\homo{A}{P(eA(1-e)A)\oplus(1-e)A}{e'A}$ belongs to $J_{\modu{A}}$.
In particular $(1-e)A$ and $e'A$ do not have a non-zero common summand,
and we have  $\add{eA}\supseteq \add{e'A}$. 
\end{proof}

%\subsection{APR tilting modules and BB tilting modules}

%Let us recall APR tilting modules and BB tilting modules. 
%We say that an $A$-module $T$ is a (classical) \emph{tilting module} if
%\begin{itemize}
%\item the projective dimension ${\rm{pd}}(T)$ of $T$  is at most $1$,
%\item ${\rm{Ext}}_{A}^{1}(T, T)=0$, 
%\item $\delta (T)=\delta (A)$. 
%\end{itemize}

We end this section by observing a connection between BB and APR tilting modules and silting mutation.

\begin{definition}\label{def:APR and BB}
Let $e\in A$ be a primitive idempotent and $S$ the corresponding simple $A$-module.
Assume ${\rm{Ext}}_{A}^{1}(S, S)=0$, ${\rm{pd}}(\tau^{-1}S)\leq 1$ and that $S$ is not injective.
%\old{that the following conditions are satisfied.
%\begin{itemize}
%\item ${\rm{Ext}}_{A}^{1}(S, S)=0$,
%\item ${\rm{pd}}(\tau^{-1}S)\leq 1$,
%\item $S$ is not injective.
%\end{itemize}}
We define a \emph{BB tilting module} \cite{BB} with respect to $e$ by
\[T:=\tau^{-1}S\oplus (1-e)A\] 
%where $V$ runs non-isomorphic simple $A$-modules except a simple $A$-module $S$.
%and $P(X)$ is a projective cover of $X$ for an $A$-module $X$. 
%Then $T$ is a tilting module by \cite{BB}.
A BB tilting module is called an \emph{APR tilting module} \cite{APR} if $S$ is projective. 
\end{definition}

%We shall show that they are special classes of the following silting complexes. 
%
%\begin{definition}
%For any $P\in \proj{A}$, we say that $\mu^{+}_{P}(A)$ is the \emph{APR silting complexes with respect to $P$}. 
%\end{definition}

%Note that APR silting complexes are a duality of Okuyama-Rickard complexes. 
%\old{Let us give basic properties of BB and APR tilting modules in our context of silting mutation.}

\begin{theorem}\label{prop:APR is APR}
Let $T$ be the BB tilting module with respect to $e$.
\begin{itemize}
\item[(a)] $T$ is isomorphic to left mutation $\mu_{eA}^+(A)$ of $A$ with respect to $eA$.
%the APR silting complex with respect to a projective cover $P(S)$ of $S$. 
\item[(b)] $T$ is a tilting $A$-module of projective dimension at most one.
\end{itemize}
\end{theorem}

%To prove this statement, we need the following property. 

%\begin{lemma}\label{lem:silt=tiltmod}
%Let $T$ be a silting complex in $\K^{\rm{b}}(\proj{A})$. We assume that $T^{i}=0$ unless $-1 \leq i \leq 0$, 
%and denote the $i$-th cohomological functor by $H^{i}$. 
%Then the following are equivalent:
%\begin{enumerate}[$(1)$]
%\item $H^{0}(T)$ is a tilting module;
%\item $H^{-1}(T)=0$. 
%\end{enumerate}
%\end{lemma}
%\begin{proof}
%Assume that $H^{0}(T)$ is a tilting module. Then we can easily show that $H^{-1}(T)$ is a projective module 
%satisfying $\Hom_{A}(H^{-1}(T), H^{0}(T))=0$. 
%Since $H^{0}(T)$ is a tilting module, we have $H^{-1}(T)=0$. 
%Conversely, we assume $H^{-1}(T)=0$. 
%we can easily show ${\rm{pd}}(H^{0}(T))\leq 1$, 
%and note that $T$ is isomorphic to $H^{0}(T)$ in $\D^{\rm{b}}(\modu{A})$. 
%Since $T$ is a silting complex, 
%\[\def\arraystretch{1.5}
%\begin{array}{rl}
%{\rm{Ext}}_{A}^{1}(H^{0}(T), H^{0}(T)) &\simeq \Hom_{\D^{\rm{b}}(\modu{A})}(H^{0}(T), H^{0}(T)[1]) \\
%                               &\simeq \Hom_{\D^{\rm{b}}(\modu{A})}(T, T[1]) \\
%                               &=0
%\end{array}\]
%Finally, $T$ is a tilting complex because $\Hom_{\K^{\rm{b}}(\proj{A})}(T, T[-1])\simeq \Hom_{A}(H^{0}(T), H^{-1}(T))=0$. 
%Hence the number of non-isomorphic indecomposable summands of $H^{0}(T)$ is equal to 
%the number of non-isomorphic simple $A$-modules. Thus we can obtain that $H^{0}(T)$ is a tilting module. 
%\end{proof}

\begin{proof}
(a) Take a minimal injective resolution of the $A$-module $S$:
\[0\longrightarrow S\longrightarrow D(Ae)\stackrel{f}{\longrightarrow}I\]
Since $\Ext^1_A(S,S)=0$, we have $I\in\add{D(A(1-e))}$.
Thus $f$ is a minimal left $\add{D(A(1-e))}$-approximation.
Applying $\nu^{-1}$, we have an exact sequence
\[eA\xrightarrow{\nu^{-1}f}\nu^{-1}I\longrightarrow\tau^{-1}S\longrightarrow0\]
with a minimal left $(1-e)A$-approximation $\nu^{-1}f$.
Since ${\rm{pd}}(\tau^{-1}S)\leq 1$ and $S$ is not injective, we have that $\nu^{-1}f$ is injective. 
Thus we have $T\simeq\nu_{eA}^+(A)$.

(b) Since $\nu^{-1}f$ is injective, so is $\Hom_{\K^{\rm{b}}(\proj{A})}((1-e)A,\nu^{-1}f)$.
Thus the assertion follows from Theorem \ref{when silting is tilting}.
\end{proof}

%\old{
%\begin{example}
%Let $A$ be the path algebra of a quiver
%$\xymatrix{
%1 \ar[r]  & 2 \ar[r] & 3 
%}$. Let us calculate left mutation of $A$. 
%\begin{itemize}
%\item[(a)] We have $\mu_{P_1}^+(A)=P_1[1]\oplus P_2\oplus P_3$, which is not a tilting complex. 
%\item[(b)] We have $\mu_{P_2}^+(A)=P_1\oplus S_1 \oplus P_3$, which is a BB tilting module with respect to $e_2$.
%\item[(c)] We have $\mu_{P_3}^+(A)=P_1\oplus P_2\oplus S_2$, which is an APR tilting module with respect to $e_3$.
%\end{itemize}
%\end{example}
%}

%%%%%%%%%%%%%%%%%%%%%%%%%%%%%%%%%%%%%%%%%%%%%%%%%%%%%%%%%%%%%%%%%%%%%%%%%%%%%%%%%%%%%%%%%%%%%%%%%%%%%%%%%%%
%%%%%%%%%%%%%%%%%%%%%%%%%%%%%%%%%%%%%%%%%%%%%%%%%%%%%%%%%%%%%%%%%%%%%%%%%%%%%%%%%%%%%%%%%%%%%%%%%%%%%%%%%%%

\section{Transitivity for piecewise hereditary algebras}

The aim of this section is to prove the following result.

\begin{theorem}\label{hereditary case}
Let $\TT=\K^{\rm b}(\proj{A})$ where $A$ is either a hereditary algebra or a canonical algebra over a field $k$.
Then the action of iterated irreducible silting mutation on $\silt\TT$ is transitive.
\end{theorem}

The idea of our proof of Theorem \ref{hereditary case} is to compare silting objects
with exceptional sequences.

\emph{In the rest of this section let $\TT$ be a triangulated category satisfying
the following property.}

\begin{assumption}\label{total Hom-finite}
For any $X,Y\in\TT$ we have $\Hom_{\TT}(X,Y[\ell])=0$ for any $|\ell|\gg0$.
\end{assumption}

\begin{definition}
Let $\TT$ be a triangulated category.
We say that an object $X\in\TT$ is \emph{exceptional} if
$\Hom_{\TT}(X,X[\neq0])=0$ and $\End_{\TT}(X)$ is a division algebra. 

We say that a sequence $(X_1,\cdots,X_n)$ of exceptional objects in $\TT$
is an \emph{exceptional sequence} if
\[\Hom_{\TT}(X_i,X_j[\Z])=0\ \mbox{ for any }\ 1\le j<i\le n.\]
We say that an exceptional sequence is \emph{full} if
$\thick(\coprod_{i=1}^nX_i)=\TT$.
We denote by $\exc\TT$ the set of isomorphism classes of full exceptional sequences in $\TT$.
\end{definition}

Clearly $\Z^n$ acts on $\exc\TT$ by
\[(\ell_1,\cdots,\ell_n)(X_1,\cdots,X_n):=(X_1[\ell_1],\cdots,X_n[\ell_n]).\]
Let $B_n$ be the \emph{braid group} generated by $\sigma_1,\cdots,\sigma_{n-1}$
with relations
\begin{eqnarray*}
\sigma_i\sigma_j=\sigma_j\sigma_i&|i-j|>1,\\
\sigma_i\sigma_{i+1}\sigma_i=\sigma_{i+1}\sigma_i\sigma_{i+1}.&
\end{eqnarray*} 
Then $B_n$ acts on $\exc\TT$ \cite{GR} as follows:
For an exceptional sequence ${\mathbf X}:=(X_1,\cdots,X_n)$ and $1\le i<n$,
define objects $L_{X_{i+1}}X_i$ and $R_{X_i}X_{i+1}$ in $\TT$ by
\begin{eqnarray*}
\xymatrix{
X_i\ar[r]&\coprod_{\ell\in\Z}D\Hom_{\TT}(X_i,X_{i+1}[\ell])\otimes_kX_{i+1}[\ell]
\ar[r]&L_{X_{i+1}}X_i[1]\ar[r]&X_i[1]
}\\
\xymatrix{
X_{i+1}[-1]\ar[r]& R_{X_i}X_{i+1}\ar[r]&\coprod_{\ell\in\Z}\Hom_{\TT}(X_{i}[\ell],X_{i+1})\otimes_kX_i[\ell]\ar[r]&X_{i+1}
}\end{eqnarray*}
and put
\begin{eqnarray*}
\sigma_i^{-1}{\mathbf X}&:=&(X_1,\cdots,X_{i-1},X_{i+1},L_{X_{i+1}}X_i,X_{i+2},\cdots,X_n),\\
\sigma_i{\mathbf X}&:=&(X_1,\cdots,X_{i-1},R_{X_i}X_{i+1},X_{i+1},X_{i+2},\cdots,X_n).
\end{eqnarray*}
The following transitivity result of exceptional sequences is well-known.

\begin{theorem}\cite{C,Rin,KM}\label{exceptional transitivity}
Let $\TT=\K^{\rm b}(\proj{A})$ where $A$ is either a hereditary algebra
or a canonical algebra over a field $k$. Then $B_n\times\Z^n$ acts on $\exc\TT$ transitively.
\end{theorem}

We shall deduce Theorem \ref{hereditary case} from Theorem \ref{exceptional transitivity}.

We have the following connection between silting objects and exceptional sequences,
asserting that any full exceptional sequence gives rise to a silting object.

\begin{proposition}\label{exceptional to silting}
Let ${\mathbf X}=(X_1,\cdots,X_n)$ be a full exceptional sequence in $\TT$.
Then there exists $a\in\Z$ such that $X_1[\ell_1]\oplus\cdots\oplus X_n[\ell_n]$
is a silting object for any integers $\ell_1\cdots,\ell_n\in\Z$
satisfying $\ell_i+a\le\ell_{i+1}$ for any $1\le i<n$.
\end{proposition}

\begin{proof}
By Assumption \ref{total Hom-finite}, there exists $a\ge0$
such that $\Hom_{\TT}(X_i,X_j[>a])=0$ and any $1\le i,j\le n$.
We shall show that $a$ satisfies the desired condition.

Let $\ell_1\cdots,\ell_n$ be integers satisfying $\ell_i+a\le\ell_{i+1}$ for any $i$.
Fix $1\le i<j\le n$. Since ${\mathbf X}\in\exc\TT$, we have
$\Hom_{\TT}(X_j[\ell_j],X_i[\ell_i+\ell])=0$ for any $\ell\in\Z$
and $\Hom_{\TT}(X_i[\ell_i],X_i[\ell_i+\ell])=0$ for any $\ell\neq0$.
On the other hand we have $\Hom_{\TT}(X_i[\ell_i],X_j[\ell_j+\ell])=0$
for any $\ell>0$ by $\ell_j+\ell-\ell_i>a$ and our choice of $a$.
Thus we have the assertion.
\end{proof}

We have the following easy observation.

\begin{lemma}\label{very silting}
Let ${\mathbf X}=(X_1,\cdots,X_n)$ be a full exceptional sequence in $\TT$
such that $M=X_1\oplus\cdots\oplus X_n$ is a silting object.
For any integers $\ell_1\le\cdots\le \ell_n$ we have the following.
\begin{itemize}
\item[(1)] $X_1[\ell_1]\oplus\cdots\oplus X_n[\ell_n]$ is a silting object.
\item[(2)] If $\ell_1\ge0$, then $X_1\oplus\cdots\oplus X_n$ and
$X_1[\ell_1]\oplus\cdots\oplus X_n[\ell_n]$ are transitive under
iterated irreducible silting mutation.
\end{itemize}
\end{lemma}

\begin{proof}
(1) Fix $i\le j$. Since ${\mathbf X}\in\exc\TT$, we have
$\Hom_{\TT}(X_j[\ell_j],X_i[\ell_i+\ell])=0$ for any $\ell\neq0$.
Since $X_1\oplus\cdots\oplus X_n$ is a silting object and $\ell_i\le\ell_j$,
we have $\Hom_{\TT}(X_i[\ell_i],X_j[\ell_j+\ell])=0$ for any $\ell>0$.
Thus we have the assertion.

(2) We use the induction on $\ell_1+\cdots+\ell_n$.
Let $\ell_0:=0$. Assume $\ell_{i-1}<\ell_i$ for some $1\le i\le n$.
Then $\Hom_{\TT}(X_j[\ell_j],X_i[\ell_i])=0$ for any $j>i$ since ${\mathbf X}\in\exc\TT$,
and also for any $j<i$ since $M\in\silt\TT$ and $\ell_i-\ell_j>0$.
Consequently we have $\Hom_{\TT}(X_j[\ell_j],X_i[\ell_i])=0$ for any $j\neq i$.
Thus we have
\[\mu_{X_i[\ell_i]}^-(X_1[\ell_1]\oplus\cdots\oplus X_n[\ell_n])=
X_1[\ell_1]\oplus\cdots X_i[\ell_i-1]\cdots\oplus X_n[\ell_n].\]
Thus the assertion follows inductively.
%
%(ii) Next we assume $\ell_1<0$.
%
%By (i), we know that $X_1\oplus\cdots\oplus X_n$ and $X_1\oplus X_2[\ell_2-\ell_1]\oplus\cdots\oplus X_n[\ell_n-\ell_1]$
%are transitive under iterated irreducible silting mutation,
%and that $X_1[\ell_1]\oplus\cdots\oplus X_n[\ell_n]$ and $X_1\oplus X_2[\ell_2-\ell_1]\oplus\cdots\oplus X_n[\ell_n-\ell_1]$ 
%are transitive under iterated irreducible silting mutation.
%Thus the assertion follows.
\end{proof}

We have the following transitivity result.

\begin{proposition}\label{first transitivity}
Let $(X_1,\cdots,X_n)$ be a full exceptional sequence in $\TT$
and $\ell_1,\cdots,\ell_n\in\Z$,
If $X_1\oplus\cdots\oplus X_n$ and
$X_1[\ell_1]\oplus\cdots\oplus X_n[\ell_n]$ are silting objects,
then they are transitive under iterated irreducible silting mutation.
\end{proposition}

\begin{proof}
Choose integers $0\le m_1\le\cdots\le m_n$ satisfying
$0\le m_1-\ell_1\le\cdots\le m_n-\ell_n$.
By Lemma \ref{very silting} we have that
$X_1[m_1]\oplus\cdots\oplus X_n[m_n]$ is an iterated irreducible silting mutation of $X_1\oplus\cdots\oplus X_n$
(respectively, $X_1[\ell_1]\oplus\cdots\oplus X_n[\ell_n]$).
Thus we have the assertion.
\end{proof}

\emph{In the rest of this section let $\TT=\K^{\rm b}(\proj{A})$
where $A$ is either a hereditary algebra or a canonical algebra over a field.}
We can identify $\TT$ with $\D^{\rm b}(\AA)$ for a hereditary abelian
category $\AA$.
Any indecomposable object in $\TT$ is isomorphic to $X[\ell]$ for some
$X\in\AA$ and $\ell\in\Z$.
Moreover $\TT$ satisfies Assumption \ref{total Hom-finite}.

Let us start with the following preliminary results.

\begin{lemma}\label{preliminaries}\cite[IV.1.2, IV.1.5]{H}
\begin{itemize}
\item[(1)] Let $X,Y\in\AA$ be indecomposable objects satisfying $\Ext^1_{\AA}(Y,X)=0$.
Then any non-zero morphism $f\in\Hom_{\AA}(X,Y)$ is either injective or surjective.
\item[(2)] Let $X,Y\in\AA$ be exceptional objects.
If $\Hom_{\AA}(X,Y)\neq0$ and $\Ext^1_{\AA}(X,Y)\neq0$, then we have $\Ext^1_{\AA}(Y,X)\neq0$.
\end{itemize}
\end{lemma}

We have the following result asserting that $\RHom_{\AA}(X,Y)$ is a stalk complex.

\begin{lemma}\label{only one extension}
Let $X,Y\in\TT$ be indecomposable objects. If $M:=X\oplus Y$ satisfies
$\Hom_{\TT}(M,M[>0])=0$,
then there exists at most one $\ell\in\Z$ such that $\Hom_{\TT}(X,Y[\ell])\neq0$.
\end{lemma}

\begin{proof}
Without loss of generality, we can assume that $X$ and $Z:=Y[a]$ belong to $\AA$.
Since $\AA$ is hereditary, we have $\Hom_{\TT}(X,Z[\neq0,1])=0$.
Assume that $\Hom_{\TT}(X,Z)\neq0$ and $\Hom_{\TT}(X,Z[1])\neq0$.
By Lemma \ref{preliminaries}(2) we have $\Hom_{\TT}(Z,X[1])\neq0$.

Since $\Hom_{\TT}(X,Z[1])\neq0$, we have $\Hom_{\TT}(M,M[1+a])\neq0$.
Thus we have $a<0$.
On the other hand, since $\Hom_{\TT}(Z,X[1])\neq0$, we have
$\Hom_{\TT}(M[a],M[1])\neq0$. Thus we have $a>0$, a contradiction.
\end{proof}

We have the following information about existence of $\ell$ in Lemma \ref{only one extension}.

\begin{lemma}\label{no cycle}
Let $X_{n+1}=X_1,\cdots,X_n\in\TT$ be pairwise non-isomorphic indecomposable objects.
Assume that the following conditions are satisfied.
\begin{itemize}
\item[(i)] $M:=X_1\oplus\cdots\oplus X_n$ satisfies $\Hom_{\TT}(M,M[>0])=0$.
\item[(ii)] There exist integers $\ell_1,\cdots,\ell_n\in\Z$ such that
$\Hom_{\TT}(X_i,X_{i+1}[\ell_i])\neq0$ for any $1\le i\le n$.
\end{itemize}
Then we have $n=1$.
\end{lemma}

\begin{proof}
By (i) we have $\ell_i\le0$ for any $i$.
We take $a_i\in\Z$ such that $X_i\in\AA[a_i]$.
Since $\Hom_{\TT}(X_i,X_{i+1}[\ell_i])\neq0$, we have
$\Hom_{\TT}(\AA[a_i],\AA[a_{i+1}+\ell_i])\neq0$.
Thus $a_{i+1}-a_i+\ell_i$ is either $0$ or $1$. We have
\[0\ge \sum_{i=1}^n\ell_i=\sum_{i=1}^n(a_{i+1}-a_i+\ell_i)
=\sum_{i=1}^n(0\ \mbox{ or }\ 1),\]
which implies $\ell_1=\cdots=\ell_n=0$ and $a_1=\cdots=a_n$.
Thus we can assume that each $X_i$ belongs to $\AA$.

Assume $n>1$ and take a non-zero morphism $f_i:X_i\to X_{i+1}$ for any $1\le i\le n$.
Let $f_{n+1}:=f_1$. Then $f_i$ is not an isomorphism and either injective or surjective
by Proposition \ref{preliminaries}(1).
Also each $f_{i+1}f_i$ is either injective or surjective again by Proposition \ref{preliminaries}(1).
So it is impossible that $f_i$ is surjective and $f_{i+1}$ is injective at the same time.
It is easy to conclude that either all $f_i$ are injective or all $f_i$ are surjective.
Thus $f_n\cdots f_1:X_1\to X_1$ is not an isomorphism and either injective or surjective.
This is a contradiction.
\end{proof}

Now we show a certain converse of Proposition \ref{exceptional to silting}
asserting that any silting object gives rise to a full exceptional sequence.
It is also possible to show this observation by applying \cite[Theorem A]{AST}.

\begin{proposition}\label{silting to exceptional}
Let $M=X_1\oplus\cdots\oplus X_n$ be a basic silting object in $\TT$.
Then we can change indices of $X_1,\cdots,X_n$ such that
$(X_1,\cdots,X_n)$ is a full exceptional sequence in $\TT$.
\end{proposition}

\begin{proof}
For each $1\le i,j\le n$, we write $X_i\le X_j$ if there exists a sequence
$X_i=X_{i_1},X_{i_2},\cdots,X_{i_m}=X_j$ such that
$\Hom_{\TT}(X_{i_a},X_{i_{a+1}}[\ell_a])\neq0$ for some $\ell_a\in\Z$.
Clearly we have $X_i\le X_i$ and that $X_i\le X_j\le X_k$ implies $X_i\le X_k$.
By Lemma \ref{no cycle} we have that $X_i\le X_j\le X_i$ implies $i=j$.
This means that $X_1,\cdots,X_n$ forms a partially ordered set.
Thus we can change indices of $X_1,\cdots,X_n$ such that $X_i\le X_j$ implies $i\le j$.
Then $(X_1,\cdots,X_n)$ forms an exceptional sequence.
\end{proof}

For an exceptional sequence ${\mathbf X}=(X_1,\cdots,X_n)$, we let
\[[[{\mathbf X}]]:=X_1\oplus\cdots\oplus X_n\in\TT.\]
The following is a main inductive step in the proof of Theorem \ref{hereditary case}.

\begin{lemma}\label{exceptional mutation by silting mutation}
Let ${\mathbf X}=(X_1,\cdots,X_n)$ be a full exceptional sequence in $\TT$
such that $[[{\mathbf X}]]$ is a silting object.
For any $1\le i<n$, there exists $\mbox{\boldmath $\ell$}\in\Z^n$ (respectively, $\mbox{\boldmath $m$}\in\Z^n$) such that
$[[(\sigma_i,\mbox{\boldmath $\ell $}){\mathbf X}]]$ (respectively, $[[(\sigma_i^{-1},\mbox{\boldmath $m$}){\mathbf X}]]$)
is a iterated irreducible silting mutation of $[[{\mathbf X}]]$.
\end{lemma}

\begin{proof}
By Lemma \ref{only one extension} there exists at most one
$a\in\Z$ such that $\Hom_{\TT}(X_i,X_{i+1}[a])\neq0$.
Then $a$ must be non-positive since $[[{\mathbf X}]]$ is silting.
If such $a$ does not exist, we let $a:=0$.

By Assumption \ref{total Hom-finite} there exists $b\le0$ such that
\begin{eqnarray}\label{vanishing 1}
\Hom_{\TT}(X_j[b],X_{i+1}[\ell])=0&\mbox{ for any }\ 1\le j<i\ \mbox{ and }\ a\le\ell\le0.
\end{eqnarray}
By Lemma \ref{very silting} we have an exceptional sequence
\[{\mathbf Y}:=(X_1[b],\cdots,X_{i-1}[b],X_i,X_{i+1},X_{i+2},\cdots,X_n)\]
such that $[[{\mathbf Y}]]$ is an iterated irreducible silting mutation of $[[{\mathbf X}]]$.
Let
\begin{eqnarray*}
{\mathbf Z}&:=&(X_1[b],\cdots,X_{i-1}[b],X_i,X_{i+1}[a],X_{i+2},\cdots,X_n),\\
\sigma_i{\mathbf Z}&=&(X_1[b],\cdots,X_{i-1}[b],R_{X_i}(X_{i+1}[a]),X_i,X_{i+2},\cdots,X_n).
\end{eqnarray*}
By our choice of $a$ and \eqref{vanishing 1} we have
\[\mu_{i+1}^{-(-a)}([[{\mathbf Y}]])=[[{\mathbf Z}]]\ \mbox{ and }\ \mu_{i+1}^{-(-a+1)}([[{\mathbf Y}]])=[[\sigma_i{\mathbf Z}]].\]
Consequently $[[\sigma_i{\mathbf Z}]]=[[(\sigma_i,\mbox{\boldmath $\ell$}){\mathbf X}]]$
is an iterated irreducible silting mutation of $[[{\mathbf X}]]$,
where we put $\mbox{\boldmath $\ell$}:=(b,\cdots,b,0,a,0,\cdots,0)$.
\end{proof}

Now we are ready to prove Theorem \ref{hereditary case}.

Let $T=X_1\oplus\cdots\oplus X_n$ and $U=Y_1\oplus\cdots\oplus Y_n$
be basic silting objects in $\TT$.
By Proposition \ref{silting to exceptional} we can assume that
${\mathbf X}=(X_1,\cdots,X_n)$ and ${\mathbf Y}=(Y_1,\cdots,Y_n)$ are
full exceptional sequences.
By Theorem \ref{exceptional transitivity} there exists
$(\sigma,\mbox{\boldmath $\ell$})\in B_n\times\Z^n$ such that ${\mathbf Y}=(\sigma,\mbox{\boldmath $\ell$}){\mathbf X}$.
Writing down $\sigma\in B_n$ as a product of $\sigma_1^{\pm1},\cdots,\sigma_{n-1}^{\pm1}$ and applying Lemma \ref{exceptional mutation by silting mutation} repeatedly,
we have $\mbox{\boldmath $m$}\in\Z^n$ such that $[[(\sigma,\mbox{\boldmath $m$}){\mathbf X}]]$ is an iterated irreducible silting mutation of $[[{\mathbf X}]]$.
By Proposition \ref{first transitivity}
$[[(\sigma,\mbox{\boldmath $m$}){\mathbf X}]]=[[(0,\mbox{\boldmath $m$}-\mbox{\boldmath $\ell $}){\mathbf Y}]]$
is an iterated irreducible silting mutation of $[[{\mathbf Y}$]].
Thus we have the assertion.
\qed

%%%%%%%%%%%%%%%%%%%%%%%%%%%%%%%%%%%%%%%%%%%%%%%%%%%%%%%%%%%%%%%%%%%%%%%%%%%%%%%%%%%%%%%%%%%%%%%%%%%%%%%%%%%
%%%%%%%%%%%%%%%%%%%%%%%%%%%%%%%%%%%%%%%%%%%%%%%%%%%%%%%%%%%%%%%%%%%%%%%%%%%%%%%%%%%%%%%%%%%%%%%%%%%%%%%%%%%

%\section{Preliminaries on t-structures in triangulated categories}
%\label{Preliminaries on t-structures in triangulated categories}

\section{Silting subcategories in triangulated categories with coproducts} 
\label{Silting subcategories in triangulated categories with coproducts}

The aim of this section is to study silting subcategories for triangulated categories which have arbitrary coproducts. 
We need to modify the definition of silting subcategories (Definition \ref{silting coproducts}) from that in section \ref{section: silting subcategories}.
The advantage of this setting is that each set of compact objects gives rise to a torsion pair (Theorem \ref{torsion pair}),
which is not the case for the setting of section \ref{section: silting subcategories}.
As an application, we deduce a result of Hoshino-Kato-Miyachi \cite{HKM}, which associates a t-structure for each silting subcategory (Corollary \ref{HKM2}).
Moreover we show that this gives a one-to-one correspondence between silting subcategories and certain t-structures (Theorem \ref{correspondence}).
%Then we study hearts of certain t-structures including those associated with silting subcategories,
%and in particular realize them as functor categories (Theorem \ref{category P}).
We also deduce a result of Pauksztello \cite{P1,P2} on co-t-structures.

\medskip
Throughout this section, let $\TT$ be a triangulated category with arbitrary coproducts.

We say that an object $X\in\TT$ is \emph{compact} 
if $\Hom_{\TT}(X,-)$ commutes with arbitrary coproducts.
We denote by $\TT^{\rm c}$ the full subcategory consisting of compact objects in $\TT$.
We say that a subcategory $\MM$ of $\TT$ is \emph{compact} if $\MM\subset\TT^{\rm c}$,
\emph{generating} if $\MM[\Z]^{\perp_{\TT}}=0$, and
\emph{skeletally small} if isomorphism classes of objects in $\MM$ form a set.

\begin{definition}\label{silting coproducts}
We say that a subcategory $\MM$ of $\TT$ is \emph{silting} if 
$\Hom_{\TT}(\MM,\MM[>0])=0$ and $\MM$ is a skeletally small compact and generating subcategory of $\TT$.
%(i.e. $\MM\subset\TT^{\rm c}$ and $\MM[\Z]^{\perp_{\TT}}=0$).
\end{definition}

In this case $\MM$ is a silting subcategory of $\TT^{\rm c}$ in the sense of Definition \ref{definition of silting}
by the following Neeman's result:
%\begin{itemize}
%\item[(ii)'] $\thick\MM=\TT^{\rm c}$.
%\end{itemize}

\begin{proposition}\cite{N}\label{Neeman}
For any compact generating subcategory $\MM$ of $\TT$, we have $\thick\MM=\TT^{\rm c}$.
\end{proposition}

\subsection{Torsion pairs induced by sets of compact objects}\label{Torsion pairs in triangulated categories induced by compact objects}

The following main result in this section asserts that each set of compact objects gives a t-structure.

\begin{theorem}\label{torsion pair}
Let $\TT$ be a triangulated category with arbitrary coproducts,
and let $\CC$ be a set of objects in $\TT^{\rm c}$.
Then $({}^\perp(\CC^\perp),\CC^\perp)$ is a torsion pair.
\end{theorem}

We need the following general observation of homotopy colimit \cite{N}.

\begin{proposition}\label{hocolim}
Let $\TT$ be a triangulated category with arbitrary coproducts.
For a sequence 
\[X_0\xrightarrow{a_0}X_1\xrightarrow{a_1}X_2\xrightarrow{a_2}\cdots\]
we put
\[1-{\rm shift}:=\left(\begin{smallmatrix}
1&0&0&0&\cdots\\
-a_0&1&0&0&\cdots\\
0&-a_1&1&0&\cdots\\
0&0&-a_2&1&\cdots\\
\vdots&\vdots&\vdots&\vdots&\ddots
\end{smallmatrix}\right):\coprod_{i\ge0}X_i\to\coprod_{i\ge0}X_i\]
and take a triangle
\[\coprod_{i\ge0}X_i\xrightarrow{1-{\rm shift}}\coprod_{i\ge0}X_i\to Y\to\coprod_{i\ge0}X_i[1],\]
where $Y$ is called the \emph{homotopy colimit}.
Then we have the following isomorphism on $\TT^{\rm c}$:
\[\Hom_{\TT}(-,Y)\simeq{\rm colim}(\Hom_{\TT}(-,X_0)\xrightarrow{a_0\cdot}
\Hom_{\TT}(-,X_1)\xrightarrow{a_1\cdot}\Hom_{\TT}(-,X_2)\xrightarrow{a_2\cdot}\cdots).\]
\end{proposition}

We are ready to prove our Theorem \ref{torsion pair}.

%\noindent{\it Proof of Theorem \ref{torsion pair}.}
For any $X\in\TT$, we only have to construct a triangle
$S_X\to X\to U_X\to S_X[1]$ in $\TT$ with $U_X\in\CC^\perp$ and $S_X\in{}^\perp(\CC^\perp)$.
This is given as follows:

\begin{proposition}\label{construction}
Fix $X=X_0\in\TT$. For each $i\ge0$ we take a triangle
\begin{equation}\label{approximation}
\xymatrix{
C_i\ar[r]^{b_i}&X_i\ar[r]^{a_i}&X_{i+1}\ar[r]&C_i[1]
}\end{equation}
such that $b_i$ is a right $\Add\CC$-approximation.
This is possible since $\CC$ forms a set.
Thus we have a sequence
\[\xymatrix@R=.4cm{
C_0\ar_{b_0}[d]&C_1\ar_{b_1}[d]&C_2\ar_{b_2}[d]\\
X_0\ar^{a_0}[r]&X_1\ar^{a_1}[r]&X_2\ar^{a_2}[r]&\cdots
}\]
of morphisms.
We denote by $\iota_i$ the inclusion $X_i\to\coprod_{i\ge0}X_i$.
We take a triangle
\begin{equation}\label{hocolim triangle}
\coprod_{i\ge0}X_i\xrightarrow{1-{\rm shift}}\coprod_{i\ge0}X_i\xrightarrow{c}U_X\to\coprod_{i\ge0}X_i[1]
\end{equation}
and for the composition $d:=c\iota_0:X\to U_X$ we take a triangle
\begin{equation}\label{torsion pair triangle}
\xymatrix{S_X\ar[r]&X\ar[r]^d&U_X\ar[r]&S_X[1].}
\end{equation}
Then we have $S_X\in{}^\perp(\CC^\perp)$ and $U_X\in\CC^\perp$.
\end{proposition}

\begin{proof}
(i) We shall show $U_X\in\CC^\perp$.

For any $i\ge0$ we have an exact sequence
\[\Hom_{\TT}(\CC,C_i)\xrightarrow{b_i\cdot}\Hom_{\TT}(\CC,X_i)\xrightarrow{a_i\cdot}\Hom_{\TT}(\CC,X_{i+1}).\]
Since $b_i$ is a right $\Add\CC$-approximation, the left map is surjective, and hence the right map is zero.
Since $\CC$ consists of compact objects, we have by Proposition \ref{hocolim}
\[\Hom_{\TT}(\CC,U_X)\simeq{\rm colim}(\Hom_{\TT}(\CC,X_0)\xrightarrow{0}
\Hom_{\TT}(\CC,X_1)\xrightarrow{0}\cdots)=0.\]
Thus we have $U_X\in\CC^\perp$.

(ii) We shall show that $d:X\to U_X$ is a left $\CC^\perp$-approximation.

Fix any $f_0:X=X_0\to Y$ with $Y\in\CC^\perp$.
Since $f_0b_0\in\Hom_{\TT}(C_0,Y)=0$, there exists $f_1:X_1\to Y$ such that $f_0=f_1a_0$
by the triangle \eqref{approximation}.
Similarly we have $f_i:X_i\to Y$ for any $i>0$ such that
\begin{equation}\label{f=af}
f_{i-1}=f_ia_{i-1}.
\end{equation}
\[\xymatrix@R=.4cm{
C_0\ar_{b_0}[d]&C_1\ar_{b_1}[d]&C_2\ar_{b_2}[d]\\
X_0\ar^{a_0}[r]\ar_{f_0}[d]&X_1\ar^{a_1}[r]\ar_{f_1}[ld]&X_2\ar^{a_2}[r]\ar^{f_2}[lld]&\cdots\\
Y
}\]
Now we consider a morphism
\[f:=(f_0\ f_1\ \cdots)\in\Hom_{\TT}(\coprod_{i\ge0}X_i,Y).\]
By \eqref{f=af} we have $f(1-{\rm shift})=0$.
By the triangle \eqref{hocolim triangle} there exists $g\in\Hom_{\TT}(U_X,Y)$
such that $f=gc$.
In particular we have $f_0=f\iota_0=gd$.

(iii) We shall show that
$\Hom_{\TT}(\coprod_{i\ge0}X_i,\CC^\perp)\xrightarrow{\cdot(1-{\rm shift})}
\Hom_{\TT}(\coprod_{i\ge0}X_i,\CC^\perp)$ is surjective.

Let $Y\in\CC^\perp$.
Fix any $f=(f_0\ f_1\ \cdots)\in\Hom_{\TT}(\coprod_{i\ge0}X_i,Y)$.
Since $f_0b_0=0$, there exists $g_1\in\Hom_{\TT}(X_1,Y)$ such that $f_0=-g_1a_0$
by the triangle \eqref{approximation}.
Since $(f_1-g_1)b_1=0$, there exists $g_2\in\Hom_{\TT}(X_2,Y)$ such that
$f_1-g_1=-g_2a_1$.
Similarly we have $g_i:X_i\to Y$ for any $i\ge0$ which makes the following diagram commutative.
\[\xymatrix@R=.4cm{
C_i\ar[r]^{b_i}&X_i\ar[r]^{a_i}\ar[d]_{f_i-g_i}&X_{i+1}\ar[r]\ar[dl]^{-g_{i+1}}&C_i[1]\\
&Y}\]
Now the morphism $g:=(0\ g_1\ g_2\ \cdots)$ satisfies $f=g(1-{\rm shift})$.

(iv) We shall show that
$\Hom_{\TT}(U_X,\CC^\perp[1])\xrightarrow{\cdot d}\Hom_{\TT}(X,\CC^\perp[1])$ is injective.

Let $Y\in\CC^\perp[1]$. Assume that $f\in\Hom_{\TT}(U_X,Y)$ satisfies $fd=0$.
Put $g=(g_0\ g_1\ \cdots):=fc\in\Hom_{\TT}(\coprod_{i\ge0}X_i,Y)$.
Then we have $g_0=fc\iota_0=fd=0$.
\[\xymatrix@R=.4cm{
&&Y\\
\coprod_{i\ge0}X_i\ar[r]^{1-{\rm shift}}&\coprod_{i\ge0}X_i\ar[r]^{c}\ar[ur]^g&U_X\ar[r]\ar[u]^f&\coprod_{i\ge0}X_i[1]\\
S_X\ar[r]\ar[u]&X\ar[r]^d\ar[u]^{\iota_0}&U_X\ar[r]\ar@{=}[u]&S_X[1]\ar[u]
}\]
Assume $g_i=0$ for some $i\ge0$. Since
\[-g_{i+1}a_i=g_i-g_{i+1}a_i=g(1-{\rm shift})\iota_i=fc(1-{\rm shift})\iota_i=0,\]
we have that $g_{i+1}$ factors through $C_i[1]$ by the triangle \eqref{approximation}.
Since $Y\in\CC^\perp[1]$, we have $g_{i+1}=0$.

Inductively we have $g=0$, so $f$ factors through $U_X\to\coprod_{i\ge0}X_i[1]$.
Now using (iii) we have $f=0$.

(v) We shall show $S_X\in{}^\perp(\CC^\perp)$.

We have an exact sequence
\[\Hom_{\TT}(U_X,\CC^\perp)\xrightarrow{\cdot d}\Hom_{\TT}(X,\CC^\perp)\xrightarrow{}\Hom_{\TT}(S_X,\CC^\perp)\xrightarrow{}
\Hom_{\TT}(U_X[-1],\CC^\perp)\xrightarrow{\cdot d[-1]}\Hom_{\TT}(X[-1],\CC^\perp).\]
By (ii) the left map $(\cdot d)$ is surjective, and
by (iv) the right map $(\cdot d[-1])$ is injective.
Thus we have $\Hom_{\CC}(S_X,\CC^\perp)=0$, and we have the assertion.
\end{proof}

As special cases of Theorem \ref{torsion pair}, we have the following results due to Beligiannis-Reiten and Pauksztello.

\begin{corollary}\label{t and co-t}
Let $\TT$ be a triangulated category with arbitrary coproducts,
and $\CC$ a set of objects in $\TT^{\rm c}$.
\begin{itemize}
\item[(a)] \cite{BR} If $\CC[1]\subset\CC$, then $({}^\perp(\CC^\perp),\CC^\perp [1])$ is a t-structure.
\item[(b)] \cite{P2} If $\CC\subset\CC[1]$, then $({}^\perp(\CC^\perp)[1],\CC^\perp)$ is a co-t-structure.
\end{itemize}
\end{corollary}

Let us apply our results to more special cases, where we can describe the category ${}^\perp(\CC^\perp)$ in a more direct way.

The first application is the following result (b) of Hoshino-Kato-Miyachi \cite{HKM}.

\begin{corollary}\label{HKM2}
Let $\TT$ be a triangulated category with arbitrary coproducts and $\MM$ a skeletally small compact subcategory satisfying $\Hom_{\TT}(\MM,\MM[>0])=0$.
\begin{itemize}
\item[(a)] We have ${}^\perp(\MM[>0]^\perp)\subseteq\MM[\le0]^\perp$.
\item[(b)] If $\MM$ is a silting subcategory, then equality in (a) holds and $(\MM[<0]^\perp,\MM[>0]^\perp)$ is a t-structure of $\TT$.
\end{itemize}
\end{corollary}

\begin{proof}
Let $\CC:=\MM[>0]$. For any $X\in\TT$, we apply Proposition \ref{construction} to get a triangle %\eqref{torsion pair triangle} 
$S_X\to X\stackrel{d}{\to} U_X\to S_X[1]$ with $S_X\in{}^\perp(\CC^\perp)$ and $U_X\in\CC^\perp$.

Applying $\Hom_{\TT}(\MM[\le0],-)$ to \eqref{approximation} for any $i\ge0$, we have an isomorphism
\[(a_i\cdot):\Hom_{\TT}(\MM[\le0],X_i)\xrightarrow{\sim}\Hom_{\TT}(\MM[\le0],X_{i+1}).\]
Since $\MM[\le0]$ is compact, Proposition \ref{hocolim} gives an isomorphism
\begin{equation}\label{key iso}
(d\cdot):\Hom_{\TT}(\MM[\le0],X)\xrightarrow{\sim}\Hom_{\TT}(\MM[\le0],U_X).
\end{equation}

(a) Let $X\in{}^\perp(\CC^\perp)$. Then the morphism $d:X\to U_X$ is zero since $U_X\in\CC^\perp$.
Since \eqref{key iso} is an isomorphism which is zero, we have $\Hom_{\TT}(\MM[\le0],X)=0$.

(b) By Theorem \ref{torsion pair} we only have to show ${}^\perp(\CC^\perp)=\MM[\le0]^\perp$.
By (a), we only have to show ${}^\perp(\CC^\perp)\supset\MM[\le0]^\perp$.
Let $X\in\MM[\le0]^\perp$.
By \eqref{key iso}, we have $\Hom_{\TT}(\MM[\le0],U_X)\simeq\Hom_{\TT}(\MM[\le0],X)=0$.
Thus $U_X\in\CC^\perp\cap\MM[\le0]^\perp=\MM[\Z]^\perp=0$, and we have $X\simeq S_X\in{}^\perp(\CC^\perp)$.
\end{proof}

The second application is Corollary \ref{Pauksztello} below, which is
a version of a result of Pauksztello \cite{P1} on co-t-structures, where the assumption $\Hom_{\TT}(\MM,\MM[1])=0$ in \cite{P1} is dropped.

We say that a subcategory $\MM$ of $\TT$ is \emph{cosilting} if 
$\Hom_{\TT}(\MM,\MM[<0])=0$ and $\MM$ is a skeletally small compact and generating subcategory of $\TT$.

We say that an additive category $\MM$ is \emph{semisimple}
if  $\End_{\MM}(X)$ is a semisimple ring for any $X\in\MM$.

\begin{corollary}\label{Pauksztello}
Let $\TT$ be a triangulated category with arbitrary coproducts and $\MM$ a skeletally small compact and semisimple subcategory satisfying $\Hom_{\TT}(\MM,\MM[<0])=0$.
\begin{itemize}
\item[(a)] We have ${}^\perp(\MM[<0]^\perp)\subseteq\MM[\ge0]^\perp$.
\item[(b)] If $\MM$ is a cosilting subcategory, then equality in (a) holds and $(\MM[>0]^\perp,\MM[<0]^\perp)$ is a co-t-structure of $\TT$.
\end{itemize}
\end{corollary}

\begin{proof}
Let $\CC:=\MM[<0]$. 
%\old{By Theorem \ref{torsion pair} we only have to show
%${}^\perp(\CC^\perp)=\MM[>0]^\perp[-1]=\MM[\ge0]^\perp$.}
For any $X\in\TT$, we apply Proposition \ref{construction} with an additional assumption
that $\Hom_{\TT}(\MM[-1],C_i)\xrightarrow{b_i\cdot}\Hom_{\TT}(\MM[-1],X_i)$ is an isomorphism for any $i\ge0$. This is possible since $\MM$ is semisimple.
Applying $\Hom_{\TT}(\MM[\ge0],-)$ to \eqref{approximation} for any $i\ge0$, we have an isomorphism
\[(a_i\cdot):\Hom_{\TT}(\MM[\ge0],X_i)\xrightarrow{\sim}\Hom_{\TT}(\MM[\ge0],X_{i+1}).\]
Since $\MM[\ge0]$ is compact, Proposition \ref{hocolim} implies that our triangle $S_X\to X\stackrel{d}{\to} U_X\to S_X[1]$ induces an isomorphism
\begin{equation}\label{key iso2}
(d\cdot):\Hom_{\TT}(\MM[\ge0],X)\xrightarrow{\sim}\Hom_{\TT}(\MM[\ge0],U_X).
\end{equation}

(a) Let $X\in{}^\perp(\CC^\perp)$. Then the morphism $d:X\to U_X$ is zero since $U_X\in\CC^\perp$.
Since \eqref{key iso2} is an isomorphism which is zero, we have $\Hom_{\TT}(\MM[\ge0],X)=0$.

(b) By Theorem \ref{torsion pair} we only have to show ${}^\perp(\CC^\perp)\supset\MM[\ge0]^\perp$.
Let $X\in\MM[\ge0]^\perp$. By \eqref{key iso2}, we have $\Hom_{\TT}(\MM[\ge0],U_X)\simeq\Hom_{\TT}(\MM[\ge0],X)=0$.
Thus $U_X\in\CC^\perp\cap\MM[\ge0]^\perp=\MM[\Z]^\perp=0$, and we have $X\simeq S_X\in{}^\perp(\CC^\perp)$.
\end{proof}

%%%%%%%%%%%%%%%%%%%%%%%%%%%%%%%%%%%%%%%%%%%%%%%%%%%%%%%%%%%%%%%%%%%%%%%%%%%%%%%%%%%%%%%%%%%%%%%%%%%%%%%%%%%
%%%%%%%%%%%%%%%%%%%%%%%%%%%%%%%%%%%%%%%%%%%%%%%%%%%%%%%%%%%%%%%%%%%%%%%%%%%%%%%%%%%%%%%%%%%%%%%%%%%%%%%%%%%

\subsection{The correspondence between silting subcategories and t-structures}\label{The correspondence between silting subcategories and t-structures}

Let $\TT$ be a triangulated category with arbitrary coproducts.
For a silting subcategory $\MM$ of $\TT$, we observed in Corollary \ref{HKM2} that we have a t-structure 
$(\TT_{\MM}^{\le0},\TT_{\MM}^{\ge0})$ given by
\begin{eqnarray*}
\TT_{\MM}^{\le0}:={}^\perp(\MM[\ge0]^\perp)=\MM[<0]^\perp\ \mbox{ and }\ \TT_{\MM}^{\ge0}:=\MM[>0]^\perp.
\end{eqnarray*}
%If $\MM$ is a silting subcategory of $\TT$, we have $\MM\subset\TT_{\MM}^{\le0}$.
%Moreover if $\MM$ is a tilting subcategory of $\TT$, we also have $\MM\subset\TT_{\MM}^{\ge0}$.
The aim of this subsection is to characterize these t-structures obtained from silting subcategories.
The following definition is suggested by Proposition \ref{recover}.

\begin{definition}
Let $(\TT^{\le0},\TT^{\ge0})$ be a t-structure in $\TT$ with the heart $\TT^0$.
Define a subcategory of $\TT$ by
\[\MM=\MM(\TT^{\le0},\TT^{\ge0}):=\TT^{\le0}\cap{}^\perp(\TT^{<0})\cap\TT^{\rm c}.\]
\begin{itemize}
\item[(a)] We say that the t-structure $(\TT^{\le0},\TT^{\ge0})$ is \emph{silting} if $\MM$ is skeletally small and generating.
\item[(b)] We say that a silting t-structure is \emph{tilting} if $\MM\subset\TT^0$.
\end{itemize}
\end{definition}

The names of these t-structures are explained by the following our main results in this subsection.

\begin{theorem}\label{correspondence}
\begin{itemize}
\item[(a)] We have mutually inverse bijections 
\[\MM\mapsto(\TT_{\MM}^{\le0},\TT_{\MM}^{\ge0})\ \mbox{ and }\ (\TT^{\le0},\TT^{\ge0})\mapsto\MM(\TT^{\le0},\TT^{\ge0})\]
between silting subcategories of $\TT$ and silting t-structures of $\TT$.
\item[(b)] These induce bijections between tilting subcategories of $\TT$ and tilting t-structures of $\TT$.
\end{itemize}
\end{theorem}

We need the following observation.

\begin{lemma}
Let $(\TT^{\le0},\TT^{\ge0})$ be an arbitrary t-structure in $\TT$ and $\MM:=\MM(\TT^{\le0},\TT^{\ge0})$. Then
\begin{eqnarray}\label{from M to T 1}
\Hom_{\TT}(\MM,\TT^{\le0}[>0])=0,\\ \label{from M to T 2}
\Hom_{\TT}(\MM,\TT^{\ge0}[<0])=0,\\ \label{vanishing of M}
\Hom_{\TT}(\MM,\MM[>0])=0.
\end{eqnarray}
\end{lemma}

\begin{proof}
By $\MM\subset{}^\perp(\TT^{<0}$), we have \eqref{from M to T 1}.
By $\MM\subset\TT^{\le0}$, we have \eqref{from M to T 2}.
By $\MM\subset\TT^{\le0}$ and \eqref{from M to T 1}, we have \eqref{vanishing of M}.
\end{proof}

Now we are ready to prove Theorem \ref{correspondence}.

(a)(i) Let $(\TT^{\le0},\TT^{\ge0})$ be a silting t-structure and $\MM:=\MM(\TT^{\le0},\TT^{\ge0})$.
Then $\MM$ is skeletally small compact and generating by definition. Thus $\MM$ is a silting subcategory by \eqref{vanishing of M}. 

We have $\TT^{\ge0}\subset\TT_{\MM}^{\ge0}$ by \eqref{from M to T 2}, and we have $\TT^{\le0}\subset\TT_{\MM}^{\le0}$ by \eqref{from M to T 1}.
Moreover we have
\begin{eqnarray*}
\TT^{\le0}={}^\perp(\TT^{>0})\supset{}^\perp(\TT_{\MM}^{>0})=\TT_{\MM}^{\le0},\\
\TT^{\ge0}=(\TT^{<0})^\perp\supset(\TT_{\MM}^{<0})^\perp=\TT_{\MM}^{\ge0}
\end{eqnarray*}
by Corollary \ref{HKM2}, so $(\TT^{\le0},\TT^{\ge0})=(\TT_{\MM}^{\le0},\TT_{\MM}^{\ge0})$ holds.

(ii) Let $\MM$ be a silting subcategory of $\TT$.
Then $(\TT_{\MM}^{\le0},\TT_{\MM}^{\ge0})$ is a t-structure of $\TT$ by Corollary \ref{HKM2}.
Let $\NN:=\MM(\TT_{\MM}^{\le0},\TT_{\MM}^{\ge0})$.
Then $\NN$ is skeletally small since $\TT^{\rm c}$ is skeletally small by Remark \ref{T is small}.
We have
\begin{eqnarray}\label{N 1}
\Hom_{\TT}(\NN,\NN[>0])=0
\end{eqnarray}
by \eqref{vanishing of M}.
Since $\MM$ is silting, we have $\MM\subset\MM[<0]^\perp=\TT_{\MM}^{\le0}$.
Since $\TT_{\MM}^{<0}=\MM[\le0]^\perp$, we have
$\Hom_{\TT}(\MM,\TT_{\MM}^{<0})=0$ and so $\MM\subset{}^\perp(\TT_{\MM}^{<0})$.
Consequently we have
\begin{equation}\label{M is in N}
\MM\subset\TT_{\MM}^{\le0}\cap{}^\perp(\TT_{\MM}^{<0})\cap\TT^{\rm c}=\NN.
\end{equation}
%If $X\in\TT$ satisfies $\Hom_{\TT}(\NN,X[\Z])=0$, then $\Hom_{\TT}(\MM,X[\Z])=0$. Thus we have $X=0$.
Since $\MM$ is generating, so is $\NN$.
By \eqref{N 1}, we have that $\NN$ is a silting subcategory of $\TT$.
In particular, $(\TT_{\MM}^{\le0},\TT_{\MM}^{\ge0})$ is a silting t-structure.

Now $\MM$ and $\NN$ are silting subcategories of $\TT^{\rm c}$ in the sense of Definition \ref{definition of silting} by Proposition \ref{Neeman}.
By \eqref{M is in N} and Theorem \ref{no inclusion} we have $\MM=\NN$.
Thus our correspondences are mutually inverse.

(b) Let $\MM$ be a tilting subcategory of $\TT$.
Then $\MM$ is contained in both $\TT_{\MM}^{\le0}$ and $\TT_{\MM}^{\ge0}$,
so we have $\MM\subset\TT_{\MM}^0$.
Since $\MM(\TT_{\MM}^{\le0},\TT_{\MM}^{\ge0})$ equals to $\MM$,
we have that $(\TT_{\MM}^{\le0},\TT_{\MM}^{\ge0})$ is a tilting t-structure.

Let $(\TT^{\le0},\TT^{\ge0})$ be a tilting t-structure and $\MM:=\MM(\TT^{\le0},\TT^{\ge0})$. Since $\MM\subset\TT^0$, we have
\[\Hom_{\TT}(\MM,\MM[<0])\subset\Hom_{\TT}(\TT^{\le0},\TT^{>0})=0.\]
Thus $\MM$ is a tilting subcategory of $\TT$.
\qed

%%%%%%%%%%%%%%%%%%%%%%%%%%%%%%%%%%%%%%%%%%%%%%%%%%%%%%%%%%%%%%%%%%%%%%%%%%%%%%%%%%%%%%%%%%%%%%%%%%%%%%%%%%%
%%%%%%%%%%%%%%%%%%%%%%%%%%%%%%%%%%%%%%%%%%%%%%%%%%%%%%%%%%%%%%%%%%%%%%%%%%%%%%%%%%%%%%%%%%%%%%%%%%%%%%%%%%%

\end{document}